\documentclass[12 pt]{amsart}
\title[Volumes of Subvarieties]{Volumes of Subvarieties of Complex Ball Quotients \\and Sparsity of Rational Points}

\author{Soheil Memariansorkhabi}\address{Dept. of Mathematics, Purdue University, West Lafayette, IN 47907, USA.}
\email{smemaria@purdue.edu}

\theoremstyle{plain}
 \newtheorem{theorem}{Theorem}[section]

\theoremstyle{definition} 
 \newtheorem{remark}[theorem]{Remark}
 
\newtheorem{Lemma}[theorem]{Lemma}
\newtheorem{Proposition}[theorem]{Proposition}
\newtheorem{Corollary}[theorem]{Corollary}
\theoremstyle{definition}
\newtheorem{Remark}[theorem]{Remark}
\newtheorem{definition}[theorem]{Definition}

\usepackage[utf8]{inputenc}

\usepackage{amsfonts}

\usepackage{mathrsfs}
\usepackage[utf8]{inputenc}
\usepackage{amsmath, amsfonts, amsthm}
\usepackage{amssymb}
\usepackage{fullpage}
\usepackage{mathtools}
\usepackage{enumitem}
\usepackage{amsmath}
\usepackage{hyperref}
\usepackage[dvipsnames]{xcolor}

\hypersetup{
    colorlinks=true,
    citecolor=blue,
    linkcolor=magenta,
    filecolor=magenta,      
    urlcolor= black,
    }

\newcommand{\R}{\mathbb{R}} 
\newcommand{\Z}{\mathbb{Z}} 
\newcommand{\Q}{\mathbb{Q}} 
\newcommand{\C}{\mathbb{C}} 
\newcommand{\Sn}{\mathbb{S}}

\newcommand{\Proj}{\mathbb{P}}
\newcommand{\Xb}{\overline{X}}
\newcommand{\dis}{\displaystyle}
\newcommand{\del}{\partial}
\newcommand{\delb}{\bar{\partial}}

\newcommand{\bi}{\big}

\newcommand{\la}{\langle}
\newcommand{\ra}{\rangle}
\newcommand{\tr}{\operatorname{tr}}
\newcommand{\sy}{\operatorname{sys}(X) }
\newcommand{\inj}{\operatorname{inj}_x(X)}

\usepackage{mathtools}
\counterwithin{equation}{section}

\usepackage{enumitem}
\setlist[enumerate,1]{label=(\roman*)}

\let\origmaketitle\maketitle
\def\maketitle{
  \begingroup
  \def\uppercasenonmath##1{} 
  \let\MakeUppercase\relax 
  \origmaketitle
  \endgroup
}

\begin{document}
\maketitle
\begin{abstract}
Let $X=\Gamma \backslash \mathbb{B}^{n} $ be an $n$-dimensional complex ball quotient by a torsion-free non-uniform lattice $\Gamma$ whose parabolic subgroups are unipotent. We prove that the volumes of subvarieties of $X$ are controlled by
the systole of $X,$ which is the length of a shortest closed geodesic of $X$.

There are a number of arithmetic and geometric consequences: the systole of $X$ controls the growth rate of rational points on $X,$ uniformly in the field of definition. Also, we obtain effective global generation and very ampleness results for multiples of the canonical bundle $K_{\overline{X}},$ where $\overline{X}$ is the toroidal compactification of $X.$ These follow from the bound we find for the Seshadri constant of $K_{\overline{X}}$ in terms of the systole.   

\end{abstract}

\section{Introduction}

Let $X=\Gamma \backslash \mathbb{B}^{n} $ be an $n$-dimensional complex ball quotient by a torsion-free lattice $\Gamma.$ The complex ball has an intrinsic Hermitian metric (Bergman metric) which induces a K\"ahler form on $X.$ This K\"ahler form also induces a K\"ahler form on a subvariety $V$ of $X.$ The volume of $V$ with respect to the induced K\"ahler form on $X$ will be called the induced K\"ahler volume and will be denoted by $\operatorname{vol}_{X}(V).$ 

The main goal of this paper is to find a uniform lower bound  
 for the induced K\"ahler and canonical volumes of all subvarieties of a non-compact ball quotient $X$ in terms of a geometric quantity of $X:$
\renewcommand*{\thetheorem}{\Alph{theorem}}
 \begin{theorem}\label{Hyperbolic Vol Int}(Theorem \ref{Bound Hyperbolic Volume})  Let $X=\Gamma \backslash \mathbb{B}^{n} $ be a complex ball quotient by a torsion-free non-uniform lattice $\Gamma$ whose parabolic stabilizers are unipotent. Let $V\subset X$ be an irreducible  subvariety of dimension $m>0.$ Then,
\begin{align} \label{hyb vol}
          \operatorname{vol}_{X}(V)
     &\ge
      \frac{ (4 \pi)^m}{m!}\sinh^{2m}\bi(\sy /2\bi),         
\end{align}
where $\operatorname{vol}_{X}(V)$ is the volume of $V$ induced by the Bergman metric on $V$ and $\sy$ is the length of a shortest closed geodesic on $X.$
 \end{theorem}
 When $X$ is a compact ball quotient, inequality \eqref{hyb vol} was proved by Hwang and To \cite{hwang1999seshadri}. Their inequality bounds the induced K\"ahler volume of subvarieties in terms of the injectivity radius of $X.$ 
 While the injectivity radius is a positive real number for compact $X$, for non-compact  $X$, it is zero, as the injectivity radius becomes arbitrarily small near the cusps.  We generalize their inequality for non-compact $X,$ under a mild assumption on the parabolic stabilizer of 
$\Gamma,$ by replacing the injectivity radius with half of the systole. For compact 
$X,$ the injectivity radius is half of the systole. However, for non-compact 
$X,$ the systole is non-zero (see Proposition \ref{positive sys}) and can be estimated using the absolute value of the trace of hyperbolic elements in $\Gamma$ (see Lemma \ref{length trace}). 
 
Note that the assumption that the parabolic stabilizers of the lattice are unipotent is mild, as it holds for every neat lattice, and any lattice $\Gamma$ admits a finite index subgroup with this property (see Selberg's lemma \cite[page 331]{ratcliffe2006foundations}). With this assumption, the variety $X$ admits a smooth projective toroidal compactification $\Xb$ whose boundary divisor $D=\Xb \setminus X$ is a disjoint union of abelian varieties with ample conormal bundle (\cite{MOK}). Bakker and Tsimerman \cite{bakker2018kodaira} proved that if the uniform depth of cusps of $X$ is sufficiently large, then the canonical bundle of the toroidal compactification  $K_{\Xb}$ is ample. We prove in Theorem \ref{systolDepth} that the systole of $X$ bounds the uniform depth of cusps from below. Therefore, if  $\sy$ is sufficiently large, then $K_{\Xb}$ is ample.  
   
   For a subvariety $V\subset \Xb$ of dimension $m>0,$ we denote the degree of $V$ with respect to the line bundle $K_{\Xb}$ by $\operatorname{deg}_{\Xb}(V):$
   $$
    \operatorname{deg}_{\Xb}(V):= K_{\Xb}^m \cdot V.
   $$
Also, we study the canonical volume of a subvariety $V$ which is an intrinsic quantity of $V$ and a priori does not depend on the ambient space $\Xb.$ Let $V'$ be a smooth variety birational to $V$ with a canonical bundle $K_{V'}.$ The canonical volume of the variety $V$ is

$$\widetilde{\operatorname{vol}}_{V}:=\limsup\limits_{b\rightarrow \infty}\dfrac{h^0(V', bK_{V'})}{b^m/m!},$$
which does not depend on the choice of $V'.$ In particular, if $V$ is an integral curve, that is, a reduced and irreducible algebraic curve, then the canonical volume of $V$ is  $2g-2,$ where $g$ denotes the genus of the curve.  The canonical volume of $V$ measures the asymptotic growth rate of the pluricanonical linear series $|bK_{V'}|.$  The canonical volume is a non-negative real number and it is positive if and only if the linear system
$|bK_{V'}|$ embeds $V'$ birationally in a projective space for a large enough $b$, i.e., $V$ is of general type. 

    We prove that the systole controls both the canonical volume of $V$ and its degree with respect to $K_{\Xb}$ in the following sense:

       \begin{theorem} \label{LargeIntersection Int}(Theorem \ref{LargeIntersection}$+$Theorem \ref{algebriac Volume}) With the same assumption on $X$ as Theorem \ref{Hyperbolic Vol Int}, let $\Xb$ be the toroidal compactification of $X$ and let $V\subset \Xb$ be a subvariety of dimension $m>0$ with $X\cap V\neq \varnothing.$ Suppose that $\sy  \ge 4\ln\bi(5n+(8\pi)^4\bi).$     
       Then the following inequalities hold:  
       \begin{align*} 
       \widetilde{\operatorname{vol}}_{V}
       &>
     (\frac{m}{4\pi})^m e^{m\sy/16},
       \\
    \operatorname{deg}_{\Xb}(V)
    &>
    (\frac{n}{4\pi})^m  e^{m\sy/16 }.
\end{align*}

\end{theorem}

Note that systole cannot decrease in a cover and for every $X$ there exists a finite cover $X'$ such that $\operatorname{sys}(X')$ is sufficiently large (see Proposition \ref{sys-growth}). As a byproduct of Theorem~\ref{LargeIntersection Int}, we observe that in a cofinal normal tower of coverings of $X$ (see Definition~\ref{cofinal def}), the canonical volume of subvarieties can be made arbitrarily large by going sufficiently far up the tower (see Proposition \ref{sys-growth}).

\subsection*{Application I: sparsity of rational points}
A smooth toroidal compactification
$\Xb$ of $X$ can be defined over a number field $F$ provided that $\Gamma$ is neat and arithmetic (see \cite{Faltings}). Combining Theorem \ref{Hyperbolic Vol Int} and Theorem \ref{LargeIntersection Int} with the determinant method (in particular \cite[Theorem 3.4]{brunebarbe2022counting}), we get that $\sy$ controls the growth rate of rational points:
\begin{Corollary}(Corollary \ref{sparsity})\label{sparsity int} 
Suppose $\Xb$ is defined on the number field $F.$ Let $\epsilon$ be a positive number and $B$ any number such that $B\ge \epsilon[F:\Q].$
\begin{enumerate} 

    \item Let $L_1=K_{\Xb}+D.$ Then, there exists a constant $c_1$ depending on $X, F$ and $\epsilon$ such that:
\begin{align*}
         \#\Big\{
x\in X(F) \mid \operatorname{H}_{L_1}(x)\le B\Big\} \le c_1 B^{\delta},
     \end{align*}
where $$\delta=\frac{[F:\Q] n(n+3)}{ \sinh^{2}\bi(\sy /2\bi) (n+1) }(1+\epsilon),$$ 
and $H_{L_1}$ is the multiplicative height (see equation \eqref{Height def} for the definition of multiplicative height). 

\item Let $L_2=K_{\Xb}$ and assume that $\sy \ge 4\ln\bi(5n+(4\pi)^4\bi).$ Then, there exists a constant $c_2$ depending on $X, F$ and $\epsilon$ such that
\begin{align*}
         \#\Big\{
x\in X(F) \mid \operatorname{H}_{L_2}(x)\le B\Big\} \le c_2 B^{\delta},
     \end{align*}
where $$\delta=\frac{4\pi [F:\Q] (n+3)}{ e^{\sy/16}}(1+\epsilon),$$ 
and $H_{L_2}$ is the multiplicative height. 

\end{enumerate}
     
\end{Corollary}
Corollary \ref{sparsity int} tells us that if we fix $n$ and $[F:\Q],$ then the growth rate of $F$-rational points decreases as $\sy$ gets larger. This aligns with the philosophy in Diophantine geometry that geometric constraints naturally govern the arithmetic properties of a variety.   
\subsection*{Application II: effective very ampleness and Seshadri constant}

   Combining Theorem \ref{LargeIntersection Int} with the results in the adjunction theory proved by Angehrn-Siu \cite{angehrn1995effective}, Kollar \cite{kollar1997singularities} and Ein-Lazersfeld-Nakamaye \cite{ein1996zero} gives effective results in global generation, very ampleness and separation of jets:
\begin{Corollary}\label{Very amplness Int}(Corollary \ref{Effective 1}) With the same $X$ and $\Xb$ as Theorem \ref{LargeIntersection Int}, suppose that 
$$\sy 
\ge 
20\operatorname{max}\{n\ln\bi((1+2n+n!)(n+1)\bi), \ln\bi(5 n+(8\pi)^4\bi) \}.$$
Then, the following hold \begin{enumerate} 
    \item $2K_{\Xb}$ is globally generated and very ample modulo $D.$
    \item $3K_{\Xb}$ is very ample.
\end{enumerate}
\end{Corollary}

Another implication of Theorem \ref{LargeIntersection Int} is the following bound on the Seshadri constant of $K_{\Xb}$:   
       \begin{Corollary}(Theorem 
 \ref{Seperation of Jets}) \label{sepration of jets and seshadri cons int}
           Suppose that $$\sy 
\ge 
20\operatorname{max}\{n\ln\bi((1+2n+n!)(n+s)\bi), \ln\bi(5 n+(8\pi)^4\bi) \}.$$
Then $2K_{\Xb}$ separates any $s$-jets and in particular for every $x\in X,$ we have
$$\epsilon(K_{\Xb},x) \ge s/2,$$
where $\epsilon(K_{\Xb},x)$ is the Seshadri constant of $K_{\Xb}$ at $x,$ as defined in Definition \ref{DefSesh}.

       \end{Corollary}       
As the boundary divisor $D$ is a disjoint union of abelian varieties, the adjunction formula gives that $K_{\Xb|D}$ is isomorphic to the conormal bundle $O_{D}(-D),$ which is always an ample bundle due to \cite{MOK}. It is classical that every ample line bundle on an abelian variety determines a positive definite Hermitian form on that abelian variety. Let $\operatorname{sys}(D)$ be the length of a shortest closed geodesic on $D$ with respect to the metric induced by the ample line bundle $K_{\Xb|D}.$
Assuming that both $\sy$ and $\operatorname{sys}(D)$ are sufficiently large relative to $n,$ we get that the bicanonical bundle $2K_{\Xb}$ is very ample:
\begin{Corollary}\label{full Very amplness Int}(Corollary \ref{full Very amplness}) With the same $X$ and $\Xb$ as Theorem \ref{Hyperbolic Vol Int}, suppose that $\operatorname{sys}(D)>2\sqrt{2n/\pi}$ and that $$\sy 
\ge 
20\operatorname{max}\{n\ln\bi(5n(1+2n+n!)\bi), \ln\bi(5 n+(8\pi)^4\bi) \}.$$  
 Then, for every $x\in \Xb$ we have
 $$
    \epsilon(K_{\Xb},x)
    \ge 
   2n,
 $$
      and in particular $2K_{\Xb}$ is very ample.  
  
\end{Corollary}

\subsection*{Previous results and comparison }
Besides the results mentioned above, the central purposes of our paper are the following technical advancements on the subject: 
\begin{itemize}
    \item[$\bullet$] Corollary \ref{sparsity int} does not follow from the main results of Ellenberg-Lawrence-Venkatesh \cite{ellenberg2023sparsity} or Brunebarbe-Maculan \cite{brunebarbe2022counting}
    or Chiu \cite{chiu2022arithmetic} on the growth rate of integral points. 
In general, bounding the growth rate of rational points on a quasi-projective variety is more difficult than bounding the growth rate of integral points. For example, on $X=\mathbb{P}_{F}^1\setminus \{0,1,\infty\},$ there are infinitely many $F$-rational points; however, there are only finitely many integral points on $X$ because of the famous theorem of Siegel (see \cite[Remark 3.3]{brunebarbe2022counting}).   \\
It is observed in \cite{ellenberg2023sparsity,brunebarbe2022counting} that if one has a control on the degree of all subvarieties, the bound on the growth rate of rational points improves in the strategy of Bombieri-Pila \cite{bombieri1989number} and Heath-Brown \cite{heath2002density}. However, to get the lower bound on the degree of subvarieties, they passed to an \'etale cover and this restricts them to get results only on the integral points, rather than rational points. The point is that when one pulls back rational points along finite \'etale maps on quasi-projective variety, the field of definition cannot be controlled, but for integral points, it can be.   
Our intrinsic approach has the advantage that it does not require passage to a cover to raise the degree of subvarieties and hence we can get the bound on the growth rate of rational points.

    \item[$\bullet$] Our results show that the hyperbolicity properties of a non-compact ball quotient can be controlled by its systole, and hence the injectivity radius of the interior and depth of cusps need not to be dealt with separately (see for example \cite{Wong} for the other approach). In our paper, this is achieved by proving that the systole gives a lower bound for both the uniform depth of cusps (Theorem \ref{systolDepth}) and injectivity radius of the thick part (section \ref{thin-thick decomposition}).

    \item[$\bullet$] Our results depend intrinsically on $X,$ and it is not required to pass to a cover of $X$ to apply them. In particular, our results apply even in the case that $X$ is not a normal cover of another variety (see for example \cite[Theorem 5]{yeung2012tower} and \cite[Corollary 1.6.]{di2014seshadri} in which the passage to a cover of $X$ is required. Indeed, these results are about a cover of $X$ rather than $X$ itself). Note that when $\Gamma$ is a maximal lattice, $X$ is not a finite cover of other locally symmetric domain.

    \item[$\bullet$] The systole can be estimated by estimating the absolute values of traces of hyperbolic elements (see Lemma~\ref{length trace}). Especially when the coverings arise from congruence relations (see Proposition~\ref{Picard Modular Surface Ex} for an example), estimating the traces of hyperbolic elements is possibly within reach, and our results then allow one to conclude the increasing hyperbolicity behavior.
    
    This kind of increasing hyperbolicity behavior has been extensively studied for locally symmetric spaces in towers of coverings arising from congruence relations, with high ramification at the cusps (see for example \cite{nadel1989nonexistence,brunebarbe2020increasing,brunebarbe2016strong,abramovich2018level,abramovich2017level,rousseau2016hyperbolicity}). However, our approach does not require high ramification of cusps, and our results apply even in cases where some cusps do not ramify. A typical example is the covering of modular curves $X_{1}(p)\to X(1),$
which has \((p-1)/2\) cusps that do not ramify (see \cite[page 26]{shimura1971introduction} and \cite{ogg1972rational} for more details). We see their higher-dimensional analogue in Proposition~\ref{Picard Modular Surface Ex}, where the systole tends to infinity as \(p \to \infty\).
\end{itemize}

The following are some of the previous effective results for pluricanonical bundles:
\begin{enumerate}
\item Yeung proved in \cite{yeung2012tower} that for a quasi-projective variety $M,$ there exists a 
finite normal cover $M'$ such that $L^2$-holomorphic sections of $K_{M'}$ give rise to a holomorphic immersion of $M'$ into some projective space.  
    \item Di Cerbo and Lombardi proved in \cite[Corollary 1.6.]{di2014seshadri} that for a smooth projective $X$ with ample $K_{X}$ and large fundamental group, there exists a normal cover $X'$ such that  $2K_{X'}$ is very ample. In \cite[Theorem 1.3]{di2015effective}, Di Cerbo and Di Cerbo also proved an effective result for the multiple of the log canonical bundle $K_{\Xb}+D$ of the toroidal compactification of ball quotient: If $m\ge (\operatorname{n}+1)^3,$ then $m(K_{\Xb}+D)$ is ample modulo $D.$

    \item Hwang proved in \cite[Proposition 2.1.]{hwang2005number} that for a non-compact ball quotient $X,$ the sections of the line bundle $\frac{n^2+3n+4}{2}K_{X^*}$ separate any two points of Siu-Yau compactification $X^*.$
    \item  For various compact locally symmetric spaces, effective very ampleness has been studied in \cite{ 
 hwang1999seshadri, yeung2018very, yeung2001effective,yeung2017canonical,yeung2000very, wang2015effective}.
    
\end{enumerate}

\subsection*{Strategy of proof}  
To prove Theorem \ref{Hyperbolic Vol Int}, we decompose $X=\Gamma \backslash \mathbb{B}^{n} $ into two disjoint parts. The first part is the thin neighborhoods around the cusps. This part consists of all points on $X$ which have a displacement less than $\sy/2$ with respect to a parabolic element in $\Gamma.$
The second part is the complement of the first part, which we call the thick part. In Proposition \ref{radofThick} we prove that the thick part is not empty. Moreover, We prove in Theorem~\ref{thick} that every subvariety of $\Xb$ which is not entirely contained in the boundary $D$ contains a point from the thick part. Therefore, using the inequality proved by Hwang-To 
we conclude that the volume of a subvariety $V$ of $X$ is controlled by the systole.

To prove Theorem \ref{LargeIntersection Int}, the main new ingredient needed is that we show that the uniform depth of cusps is controlled by the systole in Theorem \ref{systolDepth}. Combining this with an inequality proved in \cite{memarian2022positivity}, it follows that the canonical volume of a subvariety of $\Xb$ which intersects with $X$ is also controlled by the systole. The bounds on the degree of a subvariety with respect to $K_{\Xb}$ follow from Theorem \ref{Hyperbolic Vol Int} and the previous result of Bakker-Tsimerman in \cite{bakker2018kodaira}, which is restated in Theorem \ref{Tsi}.

\section*{Acknowledgements}I would like to thank my advisor Jacob Tsimerman for many enlightening discussions and for his comments on an earlier draft of this paper.  
I also thank Sai-kee Yeung for useful conversations around the effective very ampleness problem.  
Finally, I am grateful to the anonymous referee for several helpful comments that clarified the text and the proofs in this paper.

\numberwithin{theorem}{section}

\section{Background and notation}\label{back}

In this section, we collect
the necessary background and notation which will be used frequently in the sequel. We refer to \cite{goldman1999complex, parker1998volumes, kapovich2022survey,bakker2018kodaira}
for a much fuller account.
\subsection{Geometry of complex ball quotients} \label{Geometry of complex ball quotients}
The complex unit ball $\mathbb{B}^n$ is defined as 
$$\mathbb{B}^n= \{z\in \C^n\ |\ |z|^2<1\}.$$
 The complex ball $\mathbb{B}^n$ has an intrinsic Hermitian metric called Bergman metric. The holomorphic isometry group of $\mathbb{B}^n$ with respect to this metric is the projective unitary group
$$G:= \mathrm{PU}(n,1)=\dfrac{\mathrm{U}(n,1)}{\operatorname{Z}(\mathrm{U}(n,1))},$$
where the center $\operatorname{Z}(\mathrm{U}(n,1))$ can be identified with the circle group $\{\mu I:|\mu|=1\}.$
The group $G$ acts transitively on $\mathbb{B}^n$ and acts doubly transitively on the boundary sphere $\partial \mathbb{B}^n.$ The stabilizer of the center of $\mathbb{B}^n$ is $\mathrm{U}
(n).$ Every isometry $g \in G$ is continuous
on the closed ball $\overline{\mathbb{B}^n}$ and it follows from Brouwer's fixed point theorem that $g$ has a fixed point on the closed ball $\overline{\mathbb{B}^n}$. Moreover, if there is no fixed point on $\mathbb{B}^n,$ there can be at most two fixed points on the boundary sphere $\partial \mathbb{B}^n.$ Accordingly,
an isometry $g \in G$ is classified as follows:
\begin{enumerate}
   \item Elliptic: $g$ has a fixed point $z$ in $\mathbb{B}^n$. After conjugating $g$ via $h\in G$ which
sends $z$ to $0$,
$hgh^{-1} \in \mathrm{U}(n),$ 
and therefore all eigenvalues of $g$ are roots of unity. 
    \item Parabolic: $g$ has a unique fixed point in $\overline{\mathbb{B}^n}$ and this fixed point is on the boundary $\partial \mathbb{B}^n.$ Equivalently,
    $$\dis\inf_{z \in \mathbb{B}^n}\ d(z, gz) = 0,$$
where $d(\cdot,\cdot)$ denotes the Bergman metric. This infimum is not realized for a parabolic $g.$
    \item Hyperbolic: $g$ has exactly two fixed points in $\overline{\mathbb{B}^n}$ and both are in $\partial \mathbb{B}^n.$ In
particular, $g$ preserves the unique geodesic connecting these two fixed points in $\mathbb{B}^n$ and acts as a translation along this geodesic. This geodesic is called the axis of $g.$ The length of a hyperbolic isometry $g\in G$ is $$\ell(g):=\dis\inf_{z \in \mathbb{B}^n}\ d(z, gz).$$
This infimum is not zero and is realized by any point on the axis of $g.$ The work of Chen-Greenberg on the conjugacy classification of elements of $\mathrm{U}(n,1)$ (see \cite[Theorem 3.4.1]{chen1974hyperbolic}) implies that a hyperbolic isometry $g$ has two eigenvalues $re^{i\theta}$ and $r^{-1}e^{i\theta}$ with $r>1$ and $n-1$ eigenvalues with norm $1.$  

\end{enumerate}

 Let $\Gamma \subset \mathrm{PU}(n,1)$ be a torsion-free lattice whose parabolic elements are unipotent (Selberg's lemma \cite[page 331]{ratcliffe2006foundations} tells us that every lattice in $\mathrm{PU}(n,1)$ has a finite index subgroup with this property). With this property, an element $g \in \Gamma$ is hyperbolic if and only if $g$ is semi-simple. Therefore, we will denote the set of the hyperbolic elements in $\Gamma$ by $\Gamma
 _s.$

 Let $X=\Gamma \backslash\mathbb{B}^n.$  The systole of $X$ is the length of a shortest closed geodesic with respect to the Bergman metric:  
 \begin{align} \label{sys def}
     \sy :
=
\dis\inf_{g \in \Gamma_{s}} \ell(g)
=\dis\inf_{g \in \Gamma_{s}}\{d(z, gz) | z \in \mathbb{B}^n \}. \end{align}
Equivalently, the systole of $X$ is the length of a shortest hyperbolic element in $\Gamma.$ The systole $\sy$ is always positive (see Proposition \ref{positive sys}) and the infimums in \eqref{sys def} are attained as minimums (see Remark \ref{inf is min}).

Consider $x \in X$ with stabilizer $\Gamma_x$ in $\Gamma.$ Choose a fiber $\tilde{x} \in \mathbb {B}^n.$ The injectivity radius of $x$ in $X$ is defined to be
$$\operatorname{inj}_{x}(X):=\frac{1}{2} \operatorname{inf}_{\gamma \in \Gamma\setminus \Gamma_x }d(\tilde{x},\gamma \cdot \tilde{x}),$$
which is independent of choice of $\tilde{x}.$ The injectivity radius of $X$ is $\operatorname{inj}(X):= \operatorname{inf}_{x\in X }\operatorname{inj}_{x}(X).$ In the case that $X$ is compact, $\Gamma$ only has semi-simple elements and hence $\sy=inj(X)/2.$ However, this relation does not hold for a non-compact $X$ because of the parabolic elements in $\Gamma. $

\begin{remark}\label{discontinous} Since \(\mathrm{PU}(n,1)\) acts on the unit ball \(\mathbb{B}^n\) by isometries, every lattice \(\Gamma\) in \(\mathrm{PU}(n,1)\) acts discontinuously: for every \(z\in\mathbb{B}^n\) there exists a neighborhood \(U\) of \(z\) such that
\[
\{\gamma\in\Gamma\mid\ \gamma U\cap U\neq\varnothing\}
\]
is finite. This fact is well known to hold for any discrete subgroup of the isometry group, but we will only use it for lattices (see for example \cite[Theorem 5.3.5]{ratcliffe2006foundations}).
 
\end{remark}

\subsection{Siegel domain model} \label{Siegel model}
The half-plane model of the 1-dimensional complex ball quotient is generalized by the Siegel domain model in higher dimensions. In horospherical coordinates, the Siegel domain of (complex) dimension $n$ is $\Sn = \mathbb{C}^{n-1} \times \mathbb{R} \times \mathbb{R}^{+}$. The points of $\Sn$ are written as $(\zeta, v, u) \in \mathbb{C}^{n-1} \times \mathbb{R} \times \mathbb{R}^{+}$. The boundary of $\Sn$ is $H_0 \cup \{q_\infty\}$, where $q_\infty$ is a distinguished point at infinity and $H_0 = \mathbb{C}^{n-1} \times \mathbb{R} \times \{0\}$. The point with coordinates $(0, 0, 0) \in H_0$ will be denoted by $q_0$. 

To describe the topology of the boundary, we first introduce neighborhoods of $q_\infty$. A neighborhood of $q_\infty$ is $q_\infty$ together with all points $(\zeta,v,u) \in \Sn$ with $u > \tilde{u}$ for some $\tilde{u} > 0$. For a general boundary point $q \in H_0$, there exists $g \in \mathrm{PU}(n,1)$ such that $q = g(q_\infty)$. The neighborhoods of $q$ are then defined as the images under $g$ of neighborhoods of $q_\infty$. This construction provides a basis for the topology of the boundary $H_0 \cup \{q_\infty\}$.

We follow \cite{parker1998volumes} in describing 
$\mathrm{PU}(n,1)$ via the embedding of the Siegel domain as a paraboloid in $\Proj(\mathbb{C}^{n,1}).$ To do so, we should choose a Hermitian form of signature $(n,1)$ on $\Proj(\mathbb{C}^{n,1}).$ Let  
$$
J_0:=
\begin{bmatrix}
    0&0&1\\
    0&I_{n-1}& 0\\
    1&0&0\\
\end{bmatrix},
$$
 and  $Q(z,w):=w^* J_{0} z,$
where $z$ and $w$ are column vectors in $\Proj(\mathbb{C}^{n,1})$ and $^{*}$ is the Hermitian transpose, that is, transpose the matrix and complex conjugate each
of its entries.

Consider the map  $\psi: \overline{\Sn} \to \Proj(\mathbb{C}^{n,1})$ given by

\begin{align} \label{Psi map}
    \psi: (\zeta,v,u) \longrightarrow \begin{bmatrix}
        \frac{1}{2}(-||\zeta||^2-u+i v)\\
        \zeta \\
        1
    \end{bmatrix}, \  \mbox{for} \ (\zeta,v,u) \in \overline{\Sn} \backslash \{q_\infty\}; \ 
    \psi: q_\infty \longrightarrow 
    \begin{bmatrix}
        1\\
        0 \\
        0
    \end{bmatrix}.
\end{align}
The image of this map is the set of points in $\Proj(\C^{n,1})$, where the
Hermitian form $Q$ is negative. Also $\psi$ is a homeomorphism of $\partial \Sn$ onto the set of
points where $Q$ is zero.

Let $\mathrm{U}(Q)$ be the unitary group preserving the Hermitian form $Q$ (see \cite[Section~3]{falbel2009products}):  
\begin{align} \label{def: U(Q)}
    \mathrm{U}(Q) := \{\, h \in \mathrm{GL}_{n+1}(\C) \;\mid\; Q(hz,hw) = Q(z,w) 
\ \text{for all } z,w \in \C^{n+1} \,\}.
\end{align}
The condition $Q(hz,hw) = Q(z,w)$
is equivalent to $h^* J_0 h = J_0$,  
so $h \in \mathrm{U}(Q)$ if and only if $h^{-1} = J_0 h^* J_0.$  
In particular, $h \in \mathrm{GL}_{n+1}(\C)$ lies in $\mathrm{U}(Q)$ if and only if $h$ and its inverse have the form  
\begin{align}\label{MatrixPu}
h = \begin{bmatrix}
a & \tau^* & b\\
\alpha & A & \beta \\
c & \delta^* & e
\end{bmatrix},
\quad
h^{-1} = \begin{bmatrix}
\bar{e} & \beta^* & \bar{b}\\
\delta & A^* & \tau\\
\bar{c} & \alpha^* & \bar{a}
\end{bmatrix},
\end{align}
where $A$ is an $(n-1)\times (n-1)$ matrix, $a,b,c,e \in \C$, and
$\tau,\delta,\alpha,\beta$ are column vectors in $\C^{n-1}$ (see \cite[page 438]{parker1998volumes}) . 

The projective unitary group is defined by  
\[
\mathrm{PU}(Q) := \mathrm{U}(Q) \big/ \operatorname{Z}(\mathrm{U}(Q)),
\]  
where the center $\operatorname{Z}(\mathrm{U}(Q))$ can be identified with the circle group
$\{\mu I : |\mu|=1\}$.  
Every element of $\mathrm{PU}(Q)$ is represented by a matrix in $\mathrm{U}(Q)$, uniquely determined up to multiplication by a scalar $\mu \in \C$ with $|\mu|=1$.  

The holomorphic isometry group of $\mathbb S^n$ with respect to the
Bergman metric is $\mathrm{PU}(Q)$. Its action is given by matrix
multiplication of a representative in $\mathrm{U}(Q)$ on the
paraboloid model of the Siegel domain, embedded in
$\Proj(\C^{n,1})$ via the map~\eqref{Psi map}.

The following lemma easily follows: 
\begin{Lemma} \label{FixPt} Let $\gamma$ be an element of $\mathrm{PU}(Q).$  
    \begin{enumerate}
        \item (\cite[page 7]{parker1997uniform}) If $\gamma$ swaps $q_\infty$ and $q_0,$ then it has a representative $h\in\mathrm{U}(Q)$ of the form
$$
h=\begin{bmatrix}
    0&0&1/ \overline{c}\\
    0&A&0\\
    c&0&0
\end{bmatrix},
$$
where $A\in \mathrm{U}(n-1)$ and $c\in \C.$ Consequently, $h$ acts on the horospherical coordinates $(\zeta,u,v)$ via:
$$T_{h}:(\zeta,u,v)\longrightarrow \Big(\dfrac{-2 A \zeta }{c( ||\zeta||^2+u-iv)},\dfrac{-4v }{\bi|c\bi|^2 \bi| ||\zeta
||^2+u-iv \bi|^2 },  \dfrac{4 u }{\bi|c\bi|^2 \bi| ||\zeta
||^2+u-iv \bi|^2 }\Big).$$
\item If $\gamma$ fixes both $q_\infty$ and $q_0$, then it must have a representative $h \in\mathrm{U}(Q)$ of the form
$$
h=\begin{bmatrix}
    a&0&0 \\
    0&A& 0 \\
    0& 0 & 1/\bar{a}
\end{bmatrix},
$$
where $A\in \mathrm{U}(n-1)$ and $a\in \C.$

    \end{enumerate}
\end{Lemma}

\subsection{Bergman metric}
For any pair of points $z_1=(\zeta_{1},v_1,u_1)$ and $z_2=(\zeta_{2},v_2,u_2)$ in $\Sn,$ the Bergman metric is given by:
\begin{align} \label{metric}
    d(z_1,z_2)
    =
    2\cosh^{-1} \Big(\dfrac{1}{2 \sqrt{u_1 u_2}}\bi| ||\zeta_1-\zeta_2||^2 
    +u_1+u_2
    +i v_1-i v_2
    +2 i \operatorname{Im} \la \zeta_1, \zeta_2 \ra
    \bi|\Big),
\end{align}
where  $\la.,.\ra$  denotes the standard positive definite Hermitian form on $\C^{n-1}$. Since $\cosh^{-1}(x)$ is increasing, the following lower bound can be obtained for the metric :

\begin{align} \label{metricBound} 
    d\bi((\zeta_1,v_1,u_1),(\zeta_{2},v_2,u_2)\bi)
    \ge 
    2\cosh^{-1} \Big(\dfrac{|u_1+u_2|}{2 \sqrt{u_1 u_2}}\Big)
\end{align}
The holomorphic sectional curvature of this metric is $-1$ and the sectional curvature of this metric varies on $[-1,-\frac{1}{4}]$ (see \cite{goldman1999complex}). It follows that the holomorphic bisectional curvature of this metric is bounded above by $-\frac{1}{2}$ because the holomorphic bisectional curvature always can be written as the sum of two sectional curvatures.

\subsection{Toroidal compactification}
The complex ball quotient $X$ has a unique toroidal compactification $\Xb,$ which is a smooth projective variety (see \cite{MOK}). The boundary divisor of this compactification $D:=\Xb\setminus X$ is a disjoint union of abelian varieties with ample conormal bundle. The K\"ahler form of the Bergman metric on $\Sn$ is given by 
\begin{align}
    \omega_{\Sn}:= -2i\del\delb\log(u)
\end{align}
(see \cite[Lemma 2.1]{bakker2018kodaira}). Let  $\omega_{X}$ be the K\"ahler form induced by the K\"ahler form $\omega_{\Sn}.$ It follows from Mumford's work on the singular Hermitian metric \cite{mumford1977hirzebruch} that the Bergman metric on $X$ extends as a good Hermitian metric to $\Xb.$  Integration against $\omega_{X}$ on the open part represents (as a current) a multiple of the first Chern class
\begin{align} \label{ChernForm}
    \operatorname{c}_1(K_{\Xb}+D)= \frac{1}{2\pi} \frac{n+1}{2}[\omega_{X}] \in H^{1,1}(\Xb, \R),
\end{align}

where $K_{\Xb}$ is the canonical bundle of $\Xb$ (see \cite{bakker2018kodaira}). 
\subsection{Stabilizer of cusps}
We denote the parabolic stabilizer
of $q_{\infty}$ in $G$ by
$G_{\infty}.$ 
With our choice of Hermitian form, the matrices corresponding to elements of $G_{\infty}$ are upper triangular. There is an equivalent way to identify these matrices:
\begin{Lemma}(\cite{parker1998volumes})\label{Stab of inifinty} Let $\gamma$ be an element of $\mathrm{PU}(Q).$ Let $h\in \mathrm{U}(Q)$ be a representative of $\gamma$ written in the form \eqref{MatrixPu}. Then, $\gamma$ fixes $q_{\infty}$ if and only if the $c$ entry of $h$ is $0.$ 
    
\end{Lemma}
\begin{proof}
Note that 
\begin{align*}
    h\cdot q_{\infty} 
    =
    \begin{bmatrix}
a& \tau^*& b\\
\alpha &A & \beta \\
c& \delta^* & e
\end{bmatrix}
\begin{bmatrix}
        1\\
        0 \\
        0
    \end{bmatrix}
    = 
    \begin{bmatrix}
        a\\
        \alpha  \\
        c
    \end{bmatrix},
\end{align*}
and therefore $\gamma$ fixes $q_{\infty}$ projectively if and only if $c=0$ and $\alpha=0.$ Note that if the $c$ entry of $h$ is $0,$ the multiplication of the matrix of $h$ and $h^{-1}$ in the form \eqref{MatrixPu} yields that $\alpha$ (and also $\delta$) must be $0.$  
    \end{proof}

The group $G_{\infty }$ is generated by Heisenberg isometries $I_{\infty}$ and a one-dimensional torus $T$. Heisenberg isometries consist of Heisenberg Rotations $\mathrm{U}(n-1)$ and Heisenberg translations $\mathfrak{N}$. Heisenberg Rotations $\mathrm{U}(n-1)$ act on $\zeta$-coordinates of $\Sn$ in the usual way, namely by linear isometries preserving the standard Hermitian form. Heisenberg translations $\mathfrak{N} \cong \C^{n-1} \times \mathbb{R}$ act on $\zeta$ and $v$ coordinates of  $\Sn$ via
$$(\tau, t): (\zeta, v,u) \longrightarrow \bi(\zeta+\tau, v+t+ 2i\operatorname{Im}\la\tau, \zeta\ra, u\bi).$$

The element $(0,t)\in \mathfrak{N}$ will be called the vertical translation by $t$, and the subgroup of $G_\infty$ generated by vertical translations, which is isomorphic to $\mathbb{R}$, will be denoted by $V_{\infty}.$ The vertical translation $V_{\infty}$ is the center of $G_{\infty}$ and the quotient $V_{\infty}\backslash I_{\infty}$ is isomorphic to the group of unitary transformations of 
$\C^{n-1}.$

A Heisenberg translation $(\tau, t)\in \mathfrak{N}$ fixing $q_{\infty}$ has a representative $g_{\infty}\in\mathrm{U}(Q)$ and a Heisenberg translation $(\sigma, s)\in \mathfrak{N}$ fixing $q_{0}$ has a representative $g_{0}\in\mathrm{U}(Q),$ where 
\begin{align} \label{Matrix for Hisenberg}
g_{\infty}
=
\begin{bmatrix}
    1&-\tau^*&-(|\tau|+it)/2\\
    0&I&\tau\\
    0&0&1\\
\end{bmatrix},
\ \ \ 
g_{0}=\begin{bmatrix}
    1&0&0\\
    \sigma&I&0\\
    -(|\sigma|+is)/2&-\sigma^*&1\\
\end{bmatrix}.
\end{align}
 With our assumption on $\Gamma,$ all parabolic stabilizers of $q_{\infty}$ in $\Gamma,$ i.e., $\Gamma_{\infty}:=\Gamma \cap G_{\infty}$ are Heisenberg translations.

The following statement is classical, but we include it here for completeness.
\begin{Lemma}\label{existance of shortest vertical translation} The group $\Gamma_{\infty}$ contains a shortest vertical translation.

\end{Lemma}
\begin{proof}
Note that the group of Heisenberg translations is $\mathfrak{N} \cong \mathbb{C}^{\,n-1} \times \mathbb{R}$. When $n=1$, all Heisenberg translations are vertical, so we may assume $n \ge 2$. If $g_1 = (\tau_1, t_1)$ and $g_2 = (\tau_2, t_2)$, then their product in the Heisenberg translation is
$
g_1 g_2 = \big(\tau_1 + \tau_2,\; t_1 + t_2 + 2\, \operatorname{Im} \la \tau_1, \tau_2 \ra \big).
$
Therefore, their commutator is
$
[g_1, g_2] = (0,\; 4\, \operatorname{Im} \la \tau_1, \tau_2 \ra),
$
which is a vertical translation. Note that Heisenberg translations form a non-abelian (2-step nilpotent) group when $n \geq 2$. Since $\Gamma_{\infty}$ must contain $2n-1$ generators, it necessarily contains nontrivial commutators, which correspond to vertical translations.

Moreover, because the lattice is discrete, there exists a neighborhood of the identity containing no nontrivial lattice elements. This implies the existence of a shortest nontrivial vertical translation in the parabolic stabilizer.
\end{proof}

The following fact is also well known, but we include a proof for completeness:

\begin{Lemma}\label{hyperbolic stab} 
There is no hyperbolic element of $\Gamma$ fixing $q_{\infty}$. 
\end{Lemma}

\begin{proof} 
Suppose there exists a hyperbolic element $h\in \Gamma$ fixing $q_{\infty}$. 
 Let $g$ be a vertical translation fixing $q_{\infty}.$ Using \eqref{MatrixPu}, Lemma \ref{Stab of inifinty} and \eqref{Matrix for Hisenberg}, we can choose representatives $\widetilde{h}$ and $\widetilde{g}$ of $h$ and $g$ in $\mathrm{U}(Q)$ given by
\[
\widetilde{g}=
\begin{bmatrix}
1& 0& -it/2\\
0 & I & 0 \\
0& 0 & 1
\end{bmatrix}, \quad
\widetilde{h}=
\begin{bmatrix}
a& \tau^*& b\\
0 & A & \beta \\
0& 0 & e
\end{bmatrix}, \quad
\widetilde{h}^{-1}=
\begin{bmatrix}
\bar{e}& \beta^*& \bar{b}\\
0 & A^* & \tau\\
0& 0 & \bar{a}
\end{bmatrix},
\]
where $t \in \R,$ $A$ is an $(n-1)\times (n-1)$ matrix, $a,b,e \in \C$, and 
$\tau, \beta$ are column vectors in $\C^{n-1}$. 
Since $\widetilde{h}$ and $\widetilde{h}^{-1}$ are inverses, we have $a \bar{e} = 1$, and as $h$ is hyperbolic, $|a| \neq 1$. A direct computation shows that for every $m \in \Z$, 
\[
\widetilde{h}^m \widetilde{g} \, \widetilde{h}^{-m}
= \begin{bmatrix}
1 & 0 & -i |a|^m t/2\\
0 & I & 0 \\
0 & 0 & 1
\end{bmatrix},
\]
so that if $|a| > 1$ taking $m \to \infty$ (or if $|a| < 1$ taking $m \to -\infty$) leads to $\widetilde{h}^m \widetilde{g} \, \widetilde{h}^{-m} \to I$, 
contradicting the fact that $\Gamma$ is discrete.

\end{proof}

\subsection{Neighborhood of cusps}\label{Neighborhood of cusps}

A horoball centered at $q_{\infty}$ with height $\tilde{u}$ is the open set
\[
B_{\infty}(\tilde{u}) := \{ (\zeta,v,u) \in \Sn \mid u > \tilde{u} \}.
\]
The height coordinate $u$ on $\Sn$ is invariant under the action of Heisenberg rotations $\mathrm{U}(n-1)$ and Heisenberg translations $\mathfrak{N}$, and hence the horoball is invariant under these groups.

Two points on $\partial \Sn$ are considered equivalent if they lie in the same $\Gamma$-orbit. A cusp of $X$ is the equivalence class of a point on $\partial \Sn$ fixed by a parabolic element of $\Gamma$. Thus, the cusps of $X$ are in one-to-one correspondence with the $\Gamma$-orbits of parabolic fixed points on $\partial \Sn$.

The complex ball quotient $X$ has finitely many cusps, and every cusp of $X$ can, possibly after a change of coordinates, be represented by the equivalence class of $q_\infty$. More precisely, if a cusp corresponds to the class of another point on $\partial \Sn$, the transitivity of the action of $\mathrm{PU}(n,1)$ on the boundary allows us to move that point to $q_\infty$, which amounts to replacing $\Gamma$ by its conjugate in $\mathrm{PU}(n,1)$.

Let $c_i$ be a cusp of $X$ corresponding to the equivalence class of $q_\infty$, and let $\Gamma_i \subset \Gamma$ denote the parabolic stabilizer of $c_i$. The smallest $\tilde{u}$ such that 
\[
\Gamma_i \backslash B_{\infty}(\tilde{u})
\] 
injects into $X$ is called the height of the cusp $c_i$, denoted by $u_i$. By Parker's generalization of Shimizu's lemma \cite{parker1998volumes}, for sufficiently large $\tilde{u}$ the set $\Gamma_i \backslash B_{\infty}(\tilde{u})$ injects into $X$. With our identification of $c_i$ with the $\Gamma$-orbit of $q_\infty$, we have $\Gamma_i = \Gamma_\infty$. For every $\tilde{u} < u_i$, the horoball around the cusp $c_i$ with height $\tilde{u}$ is defined as
\[
B_i(\tilde{u}) := \Gamma_i \backslash B_{\infty}(\tilde{u}).
\]

Let $t_i$ be the length of a shortest vertical translation in $\Gamma_i$ (see Lemma \ref{existance of shortest vertical translation}). 
The number $d_i = t_i / u_i$ is called the depth of the cusp $c_i.$ 
Note that this quantity is invariant under conjugating $\Gamma$, and hence is well-defined independently of the choice of coordinates.

\begin{definition}\label{UniDep}(\cite[Definition 3.7.]{bakker2018kodaira}) The uniform depth of the cusps of $X$ is the largest $d$ satisfying the following properties:
\begin{enumerate}
     \item for every $i$, $d\le d_i$ (this gives that $\Gamma_i \backslash B_{i}(t_{i}/d)$ injects into $X$).     
     \item all $\Gamma_i \backslash  B_{i}(t_{i}/d)$ are pairwise disjoint.  
\end{enumerate}

\end{definition}

\section{Systole and depth of cusps}\label{Sec2}
Let $X=\Gamma \backslash \mathbb{B}^n,$ where $\Gamma\subset \mathrm{PU}(n,1)$ is a torsion-free lattice whose parabolic stabilizers are unipotent. In this section we frequently use the content and notation introduced in section \ref{back}.
The main goal of this section is to prove Theorem \ref{systolDepth},
where we show that the systole $\sy$ bounds the uniform depth of cusps $d$ of $X$ from below.  

To see the relation between the systole and depth of cusps, we first prove that the length of a hyperbolic element in $\Gamma$ only depends on its non-unit eigenvalues:     
\begin{Proposition} \label{Length}
   Suppose $h\in \Gamma$ is a hyperbolic element. Let $\hat{h}\in\mathrm{U}(Q)$ be a representative of $h$    with non-unit eigenvalues $re^{i\theta}$ and $r^{-1} e^{i\theta}$. Then,
    $$\ell(h)= 2|\ln(r)|.$$
\end{Proposition} 

\begin{proof}
Since $h$ is hyperbolic, it fixes two distinct points $x_1$ and $x_2$ on the boundary $\partial \Sn.$ As $\mathrm{PU}(Q)$ acts doubly transitive on the boundary, there exists $P \in \mathrm{PU}(Q)$ such that $P(x_1)=q_0$ and $P(x_2)=q_\infty.$ Now we can write
\begin{align*}
    d(x, hx)= d(Px, PhP^{-1} Px)= d(x', PhP^{-1} x'), 
\end{align*}
where $x'= Px.$ Suppose $x'=(\zeta_{1}, v_1,u_1),$ and $PhP^{-1} x'= (\zeta_2, v_2, u_2).$ Since $PhP^{-1}$ fixes both $q_0$ and $q_\infty,$ it follows from Lemma \ref{FixPt} that it has a representative $\widetilde{h}\in\mathrm{U}(Q)$ such that 
$$\widetilde{h}=\begin{bmatrix}
    a&0&0 \\
    0&A& 0 \\
    0& 0 & 1/\bar{a}
    \end{bmatrix},$$
for a complex number $a$ and $A\in \mathrm{U}(n-1).$ Therefore, in horospherical coordinates, using the map \eqref{Psi map}, we obtain the following: 
\begin{align*}
    \widetilde{h}\cdot x' =\begin{bmatrix}
    a&0&0 \\
    0&A& 0 \\
    0& 0 & 1/\bar{a}
    \end{bmatrix}
    \begin{bmatrix}
        \frac{1}{2}(-||\zeta_{1}||^2-u_1+i v_1)\\
        \zeta_{1} \\
        1
    \end{bmatrix}
    =
    \begin{bmatrix}
        \frac{a}{2}(-||\zeta_{1}||^2-u_1+i v_1)\\
        A \zeta_{1} \\
        1/\bar{a}
    \end{bmatrix}.
\end{align*}
This gives $\zeta_2= \bar{a} A\zeta_1$ and
\begin{align} \label{firstcoe1}
    \frac{1}{2}(-||\zeta_2||^2-u_2+iv_2)
    =
    \frac{|a|^2}{2}(-||\zeta_1||^2-u_1+iv_1).
\end{align}
Therefore, $u_2= |a|^2 u_1$ and $v_2= |a|^2 v_1.$ Note that as conjugation does not change the eigenvalues, we have that $|a|^2=r^2 \ \mbox{or} \ \frac{1}{r^2}.$  On the other hand, inequality \eqref{metricBound} yields that
 \begin{align}\nonumber
 d\bi(x', PhP^{-1} x' \bi)
    &\ge 
    2\cosh^{-1} \Big(\dfrac{|u_1+u_2|}{2\sqrt{u_1 u_2}}\Big)\\ \nonumber
    &\ge
    2\cosh^{-1} \Big(\frac{1}{2}(r+\frac{1}{r})\Big).
\end{align}

Since this lower bound is realized at $\zeta_1=0, v_1=0,$  and $\cosh^{-1} \Big(\frac{1}{2}(r+\frac{1}{r})\Big)=|\ln(r)|,$ we can conclude that the equality holds. 
\end{proof}

\begin{definition}\label{well defined rep}
Let $\gamma \in \Gamma \subset \mathrm{PU}(Q)$ and choose a lift 
$\widetilde{\gamma} \in \mathrm{U}(Q)$ representing it.
\begin{enumerate}
    \item  The \emph{absolute trace} of $\gamma$ is defined as
    \begin{align*}
        |\operatorname{tr}(\gamma)| := \bigl|\operatorname{tr}(\widetilde\gamma)\bigr|.
    \end{align*}
    \item \label{def: absolute $c$-entry} The \emph{absolute $c$-entry} of $\gamma$ is defined as the absolute value of the lower-left entry of the matrix $\widetilde{\gamma}$, written in the form \eqref{MatrixPu}.
\end{enumerate}
\end{definition}

Both notions are well-defined, since any two lifts differ by a scalar in 
$\mathrm{U}(1)$, which does not affect the absolute value of either the trace or the $c$-entry. For $\gamma \notin \Gamma_{\infty},$
we know by Lemma \ref{Stab of inifinty} that the absolute $c$-entry 
cannot be zero. 

\begin{definition}\label{c inf def}
For a torsion-free lattice $\Gamma \subset \mathrm{PU}(Q)$, we call the infimum of the absolute $c$-entries (see Definition \ref{well defined rep} (ii)) among all $\gamma \in \Gamma \setminus \Gamma_{\infty}$ the \emph{infimum $c$-entry} of $\Gamma$, and we denote it by $c_{\inf}$.
\end{definition}

Parker's generalization of Shimizu's lemma gives the following:

\begin{remark}\label{min c}  
By \cite[Theorem~2.3]{parker1998volumes}, one has $c_{\inf} > \tfrac{4}{t_{\infty}}$, where $t_{\infty}$ is the length of a shortest vertical translation fixing $q_\infty.$  
\end{remark}

\begin{Lemma}\label{length trace}
Let $\gamma\in \mathrm{PU}(Q)$ be a hyperbolic element. Then
\[
\ell(\gamma)\ge 2\ln\bigl(\frac{1}{2}(|\operatorname{tr}(\gamma)|-n+1)\bigr).
\]
\end{Lemma}

\begin{proof}
Let $\widetilde{\gamma}$ be any lift of $\gamma$ to $\mathrm{U}(Q)$. 
Suppose $re^{i\theta}$ and $r^{-1}e^{i\theta}$ are the non-unit eigenvalues of $\widetilde{\gamma}$ with $r>1.$
Since $\widetilde{\gamma}$ has $n-1$ remaining eigenvalues, all of which are units, we obtain
\[
2r \ge r+\frac{1}{r} \ge |\operatorname{tr}(\widetilde{\gamma})| - n + 1.
\]
Finally, since $|\ln(x)|$ is increasing for $x>1$, we conclude from Proposition \ref{Length} that the desired inequality holds.
\end{proof}
 
Lemma \ref{length trace} tells us that if for every hyperbolic element $\gamma \in \Gamma$ the quantity $|\operatorname{tr}(\gamma)|$ is sufficiently large, then the systole $\sy$ will also be large; in other words, the systole $\sy$ can be estimated by estimating the absolute values of the traces of hyperbolic elements (see Proposition~\ref{Picard Modular Surface Ex} for an example of this estimation).

We recall a lemma from Parker's version of Shimizu's lemma:  
\begin{Lemma}(see \cite[Lemma 2.6]{parker1998volumes}) \label{trace}
    Let $g_{\infty}=(0,t)$ be a vertical translation fixing $q_{\infty}$ and let $h$ be an element of $\mathrm{PU}(Q)$ whose absolute $c$-entry is $c_{h}.$ Then,
    $$
    \big|\tr[g_{\infty},h]\big|=n+1+ \frac{1}{4}|c_ht|^2.
    $$
    
\end{Lemma}
\begin{proof} Let $\widetilde{g}$ and $\widetilde{h}$ be representatives
of $g_{\infty}$ and $h$ in $\mathrm{U}(Q)$, given by
$$
\widetilde{g}=\begin{bmatrix}
1& 0& -i\frac{t}{2}\\

0 &I & 0 \\

0& 0 & 1
\end{bmatrix}, \
\widetilde{h}= \begin{bmatrix}
a& \tau^*& b\\
\alpha &A & \beta \\
c& \delta^* & e
\end{bmatrix}, \
\widetilde{h}^{-1}=\begin{bmatrix}
\bar{e}& \beta^*& \bar{b}\\
\delta & A^* &  \tau\\
\bar{c}& \alpha^* & \bar{a}
\end{bmatrix},
$$
where \eqref{MatrixPu} was used to find the inverse of $\widetilde{h}$, and necessarily $|c| = c_h$. As noted after the definition of the absolute trace \eqref{well defined rep}, to compute 
$|\operatorname{tr}[g_{\infty},h]|$ it suffices to evaluate 
$|\operatorname{tr}[\widetilde{g},\widetilde{h}]|$ (it does not depend on the choice of representatives).    

To find $|\tr[\widetilde{g},\widetilde{h}]|$, note that  
$$\widetilde{g}\widetilde{h}=\begin{bmatrix}
a-\frac{i}{2} ct&\tau^*-\frac{i}{2} t\delta^* & b-\frac{i}{2}te\\

\alpha &A & \beta \\

c& \delta^* & e
\end{bmatrix},\
\widetilde{g}^{-1} \widetilde{h}^{-1}=\begin{bmatrix}
\bar{e}+\frac{i}{2} \bar{c}t&\beta^*+\frac{i}{2} t\alpha^* & \bar{b}+\frac{i}{2}t\bar{a}\\

\delta & A^* &  \tau \\

\bar{c}& \alpha^* & \bar{a}
\end{bmatrix}.
$$
Therefore, using the relation given by equation \eqref{MatrixPu} it follows that
\begin{align*}
\operatorname{tr} [\widetilde{g},\widetilde{h}]
&=
a\bar{e}+ \frac{1}{4}|ct|^2-\frac{it}{2}(c\bar{e}-\bar{c}a)
+
\delta \tau^* - \frac{i}{2} t|\delta|^2
+
b\bar{c} -\frac{i}{2}te\bar{c} \\ \nonumber
& \hspace{.4cm}+
\alpha \beta^* +\frac{i}{2}t|\alpha|^2+AA^* + \beta \alpha^* 
+
c\bar{b}+\frac{i}{2}t\bar{a}c +\delta^* \tau + e \bar{a} \\ \nonumber
&=
n+1+ \frac{1}{4}|ct|^2.
\end{align*}
\end{proof}

We prove a lemma which will help us to see the relation between the depth of a cusp in terms of the trace of the hyperbolic elements in $\Gamma:$
\begin{Lemma}\label{shim} Let $\gamma \in \Gamma \setminus \Gamma_{\infty}$ have absolute $c$-entry $c_{\gamma}$. 
For every $z \in \mathbb{S}$ the following inequality holds:
\[
u(z)\,u(\gamma \cdot z) \;\le\; \frac{4}{|c_{\gamma}|^{2}}.
\]

\end{Lemma}
\begin{proof}
 There are unique Heisenberg transformations $h_1, h_2$ such that $h_1(q_0)=\gamma(q_\infty)$ and $h_2^{-1}(q_0)=\gamma^{-1}(q_\infty).$ Consider $\hat{\gamma}= h_1^{-1} \gamma h_2^{-1}$ and note that as the Heisenberg translations are stabilizers of the $u$-coordinate, we have that 
$u(\gamma z)=u(\hat{\gamma} z).$ Because both $h_1^{-1}$ and $h_2^{-1}$ fix $q_{\infty}$, it follows from Lemma \ref{Stab of inifinty} that the absolute $c$-entries of $\gamma$ and $\hat{\gamma}$ coincide. Moreover, since both $h_1^{-1}$ and $h_2^{-1}$ fix $q_\infty,$ the element $\hat{\gamma}$ swaps $q_\infty$ and $q_0.$ Therefore, Lemma \ref{FixPt} tells us that $\hat{\gamma}$ has a representative $\tilde{\gamma}\in \mathrm{U}(Q)$ such that its acts on the horospherical coordinates $(\zeta,u,v)$ via: 
$$T_{\tilde{\gamma}}:(\zeta,v,u)\longrightarrow \Big(\dfrac{A \zeta r_{\gamma}^2}{||\zeta
||^2+u-iv},\dfrac{-v r_{\gamma}^4}{\bi| ||\zeta
||^2+u-iv \bi|^2 },  \dfrac{u r_{\gamma}^4}{\bi| ||\zeta
||^2+u-iv \bi|^2 }\Big),$$
where $A\in \mathrm{U}(n-1)$ and $r_{\tilde{\gamma}}= \sqrt{\dfrac{2}{|c_{\gamma}|}}.$ This gives that 
$$u(z) u(\gamma\cdot z)= u(z) u(\tilde{\gamma}\cdot z)
= \dfrac{u^2}{|||\zeta||^2+u-iv|^2}\cdot \bi|\dfrac{2}{c_{\gamma}}\bi|^2 \le \dfrac{4}{|c_{\gamma}|^2}.$$
\end{proof}

Lemma \ref{shim} implies the following:
\begin{Proposition}\label{height} 
Let $c_{\inf}$ be the infimum of the absolute values of the $c$-entries of $\Gamma$, as defined in Definition \ref{c inf def}.  
The horoball $B_{\infty}(2/c_{\inf})$
injects into $X$, and therefore the depth of a cusp associated with the equivalence class of $q_{\infty}$ is at least $\dfrac{t_{\infty} \cdot c_{\inf}}{2}.$
\end{Proposition}

\begin{proof}
Let $\gamma \in \Gamma \setminus \Gamma_{\infty}$, and consider the horoball centered at $q_{\infty}$ with height $2/c_{\inf}:$
\[
U_{\infty} \;=\; \{ z \in \Sn \mid u(z) > \tfrac{2}{|c_{\inf}|} \}.
\] By Lemma \ref{shim}, for every $z \in \Sn$ we have
\[
u(z)\,u(\gamma \cdot z) \;\le\; \frac{4}{|c_{\inf}|^{2}}.
\]
Hence, the sets $U_\infty$ and $\raisebox{0.3ex}{$\gamma$} \, U_{\infty}$ are disjoint. This implies that the horoball
$B_{\infty}(2/c_{\inf}) = \Gamma_{\infty} \backslash U_{\infty}$ injects into $X$. By the definition of the depth of a cusp, the depth of the cusp associated with the equivalence class of $q_{\infty}$ is at least $\dfrac{t_{\infty} \cdot c_{\inf}}{2}.$  
\end{proof}

Now we will show that the systole is always positive:

\begin{Proposition} \label{positive sys}
The systole, as defined in \eqref{sys def}, is always positive. 
\end{Proposition} 

\begin{proof} 
Suppose not. Then there would exist a sequence of hyperbolic elements $\{h_m\}_{m=1}^{\infty}$ in $\Gamma$ such that $\ell(h_m) \to 0$ as $m \to \infty.$ Let $r_m e^{i\theta_m}$ be the eigenvalue of a representative of $h_m$ in $\mathrm{U}(Q)$ with $r_m>1.$ It follows from Proposition \ref{Length} that $r_m \to 1$ as $m \to \infty.$ By the definition of the infimum, there exists $z_m \in \mathbb{B}^n$ such that 
\[
d(z_m, h_m \cdot z_m) \le \ell(h_m) + \frac{1}{m}.
\]
Therefore $d(z_m, h_m \cdot z_m) \to 0$ as $m \to \infty.$ Since $\overline{\mathbb{B}^n}$ is compact, there exists a subsequence of $\{z_m\}_{m=1}^{\infty}$, again denoted by $\{z_m\}_{m=1}^{\infty}$, which converges to a point $z \in \overline{\mathbb{B}^n}.$ By the triangle inequality,
\[
d(z, h_m \cdot z) \le d(z,z_m) + d(z_m, h_m \cdot z_m) + d(h_m \cdot z_m, h_m \cdot z),
\]
and since $d(h_m \cdot z_m, h_m \cdot z) = d(z_m,z),$ we get $d(z,h_m \cdot z) \to 0$ as $m \to \infty.$
 
Now consider two cases:  
\begin{enumerate}
    \item $z \in \mathbb{B}^n:$ In this case, the set $\{h_m \cdot z\}_{m=1}^{\infty}$ has an accumulation point in the interior of the unit ball, contradicting the fact that $\Gamma$ acts discontinuously on $\mathbb{B}^n$(see Remark \ref{discontinous}).
    
    \item $z \in \partial \mathbb{B}^n:$ Since $\mathrm{PU}(n,1)$ acts transitively on the boundary, there exists $g \in \mathrm{PU}(n,1)$ such that $g \cdot z = q_\infty.$ Consider the sequences $\{h'_m := g h_m g^{-1}\}$ in the lattice $g \Gamma g^{-1}$ and set $z'_m := g \cdot z_m.$ Note that $z'_m \to q_\infty$ and $d(z'_m, h'_m \cdot z'_m) \to 0$ as $m \to \infty$; hence, $u(z'_m) \to \infty$ and $u(h'_m \cdot z'_m) \to \infty.$ Let $c_m$ be the absolute $c$-entry of $h'_m.$  Note that Lemma \ref{hyperbolic stab} tells us that every non-identity element in a cusp stabilizer is parabolic, therefore, none of the $h'_m$ fixes $q_\infty.$ By Lemma \ref{Stab of inifinty}, we have $c_m \neq 0$, and by Remark \ref{min c}, $c_m \ge 4/t_\infty$ for all $m$, where $t_\infty$ is the length of a shortest vertical translation around $q_\infty.$ Then Lemma \ref{shim} gives
    \[
    u(z'_m) \, u(h'_m \cdot z'_m) \le \frac{4}{|c_m|^2} \le \frac{t_\infty}{4}.
    \]
   This inequality contradicts the fact that both $u(z'_m)$ and $u(h'_m \cdot z'_m)$ tend to infinity as $m \to \infty.$
\end{enumerate}
\end{proof}

Consider the set 
$$S_{\Gamma}:=\big\{ \gamma \in \Gamma \big| |\tr(\gamma)|> n+1 \big\},$$
associated to $\Gamma.$ It follows from the classification of isometries that if $\gamma \in \Gamma$ has $|\operatorname{tr}(\gamma)|> n+1,$ then $\gamma$ must be hyperbolic. Hence, all elements of $S_{\Gamma}$ are hyperbolic. Also note that Lemma \ref{trace} tells us that $S_{\Gamma}$ is not empty.  
We associate the number
$$\lambda_{\Gamma}:=\dis \inf_{\gamma \in S_{\Gamma}} |\tr(\gamma)|,$$
to $\Gamma.$ Since $S_{\Gamma}$ is nonempty, it follows that $\lambda_{\Gamma}\ge n+1.$
The quantity $\lambda_{\Gamma}$ will play a role as an intermediate quantity to relate the systole of $X$ to the depth of cusps of $X$.  Specifically, we can see how $\sy$ gives a lower bound for $\lambda_{\Gamma}$:

\begin{Proposition}\label{Trace Sys ineq}
The following inequality holds:
      $
  \lambda_{\Gamma}
>
1-n+
 e^{\sy/2 }. 
$
\end{Proposition} 
\begin{proof}
    Consider $\gamma \in S_\Gamma.$ Let $\widetilde{\gamma}$ be a representative of $\gamma$ in $\mathrm{U}(Q).$ Let $re^{i\theta}$ and $r^{-1} e^{i\theta}$ be eigenvalues of $\widetilde{\gamma}$ which are not units. As $\sy$ is the length of a shortest geodesic, Proposition \ref{Length} implies that $2|\ln(r)|
    \ge
   \sy.$
    Since the other $n-1$ eigenvalues of $\widetilde{\gamma}$ have norm $1,$ the desired inequality follows from the triangle inequality. 
   
\end{proof}

Now, we can show the relation between the quantity $\lambda_{\Gamma}$ and depth of each cusps of $X:$ 

\begin{Proposition} \label{injectionHoroball}
The  depth of each cusp of $X=\Gamma \backslash \mathbb{B}^n$ is at least $\sqrt{\lambda_{\Gamma} -n -1}.$
    
\end{Proposition}
\begin{proof}
Since both the depth of cusps and $\lambda_{\Gamma}$ are invariant under the conjugation by an element of $\mathrm{PU}(n,1)$, it is sufficient to prove the lemma for a cusp $c_i$ associated with the equivalence class of $q_{\infty}.$
    Let $g_{\infty}=(0,t_{\infty})$ be a shortest vertical translation in $\Gamma_{\infty}.$ Suppose that $h\in \Gamma$ is an element which does not fix $q_{\infty}.$ Let $c$ be the absolute $c$-entry of $h,$ which is not zero by Lemma \ref{Stab of inifinty}. It follows from Lemma \ref{trace} that
    $$|\operatorname{tr}[g_\infty, h]|= n+1+ \bi|\dfrac{t_{\infty}c}{2}\bi|^2.$$
    Since $c\neq 0,$ we have that $ [g_\infty, h] \in S_{\lambda}.$ This implies that 
    $$\bi|\dfrac{t_{\infty}c}{2}\bi|\ge \sqrt{\lambda_{\Gamma} -n-1}.$$
    Since this inequality holds for every $h\in\Gamma \setminus \Gamma_{\infty},$ we can conclude that
    $$\bi|\dfrac{t_{\infty}c_{\inf}}{2}\bi|\ge \sqrt{\lambda_{\Gamma} -n-1}.$$
    Hence, Proposition \ref{height} implies $d_{i} \ge  \sqrt{\lambda_{\Gamma} -n -1},$ where $d_{i}$ is the depth of cusp $c_i.$
\end{proof}
To pass from the individual depth of cusps to the uniform depth of cusps we will use this lemma:
\begin{Lemma}(\cite[Lemma 2.5]{parker1998volumes}) \label{disjoint1} Let $B_{0}(\tilde{u}_{0})$ be the horoball of height $\tilde{u}_0$ based at $q_0$, and let $B_{\infty}(\tilde{u}_{\infty})$ be the horoball of height $\tilde{u}_{\infty}$ based at $q_\infty.$ These two horoballs are disjoint if and only if
$$\tilde{u}_0 \cdot \tilde{u}_{\infty}\ge 4.$$
\end{Lemma}

 \begin{Proposition}\label{Depth Trace Ineq} Let $d$ be the  uniform depth of cusps of $X$. Then,  
 $$d \ge \operatorname{min}\{(\lambda_{\Gamma}-n-1)^{\frac{1}{4}}, (\lambda_{\Gamma}-n-1)^{\frac{1}{2}}\}.$$
\end{Proposition}

\begin{proof} Let $d'=\operatorname{min}\{(\lambda_{\Gamma}-n-1)^{\frac{1}{4}}, (\lambda_{\Gamma}-n-1)^{\frac{1}{2}}\}.$ We will show that the horoballs $\Gamma_i \backslash B_i(t_i/d')$ inject into $X$ and they are disjoint. Since the uniform depth of cusps is the largest number satisfying these properties, the claim follows.

By Proposition \ref{injectionHoroball} we know that the depth of each cusp is at least $(\lambda_{\Gamma}-n-1)^{\frac{1}{2}},$ therefore the horoballs $\Gamma_i \backslash B(t_i/d')$ inject into $X.$ Hence, it is enough to show that for $i\neq j,$ the horoballs $\Gamma_i \backslash B(t_i/d')$ and $\Gamma_j \backslash B(t_j/d')$ are disjoint.  

Since $\mathrm{PU}(Q)$ acts doubly transitively on the boundary, we can, with a change of coordinates if necessary (i.e., by conjugating the lattice), assume that $q_i = q_\infty$ and $q_j = q_0$. Note that as both $\lambda_{\Gamma}$ and $d$ are invariant under conjugation, this change of coordinates does not change them. Let $g_0=(0, t_0)$ be the shortest vertical translation based at $q_{0}$ with $t_0>0$ and $g_{\infty}=(0,t_{\infty})$ be the shortest vertical translation based at $q_{\infty}$ with $t_{\infty}>0.$
Let $\widetilde{g}_{\infty}$ and $\widetilde{g}_{0}$ be the representatives of $g_{\infty}$ and $g_{0}$ in $\mathrm{U}(Q)$ written in the form \eqref{Matrix for Hisenberg}:
    \begin{align*}
    \widetilde{g}_{0}
    =
    \begin{bmatrix}
        1&0&0\\
        0&I&0\\
        -it_{0}/2&0&1
    \end{bmatrix},
    \ \ \ \
    \widetilde{g}_{\infty}
    =
 \begin{bmatrix}
        1&0&-it_{\infty}/2\\
        0&I&0\\
        0&0&1
    \end{bmatrix}.
     \end{align*}
    Lemma \ref{trace} implies that $|\operatorname{tr}[g_{\infty},g_{0}]|= n+1+\bi|\frac{t_{0}t_{\infty}}{4}\bi|^2.$ Therefore, $[g_{\infty},g_{0}] \in S_{\lambda}$ and it follows that 
    $$t_{0}t_{\infty} \ge 4 \sqrt{\lambda_{\Gamma}-n-1}.$$
Consider $\tilde{u}_{0}:=\frac{t_0}{(\lambda_{\Gamma}-n-1)^{\frac{1}{4}}}$ and $\tilde{u}_{\infty}:=\frac{t_\infty}{(\lambda_{\Gamma}-n-1)^{\frac{1}{4}}}.$ The inequality above implies that 
    \begin{align} \label{disjoint}
        \tilde{u}_0 \cdot \tilde{u}_{\infty}\ge 4,
    \end{align}
    and therefore it follows from Lemma \ref{disjoint1} that the horoball centered at $q_0$ with height $\tilde{u}_{0}$ and the horoball centered at $q_\infty$ with height $\tilde{u}_{\infty}$ are disjoint. Therefore, since $d'\le (\lambda_{\Gamma}-n-1)^{\frac{1}{4}},$ the horoballs $\Gamma_i \backslash B(t_i/d')$ and $\Gamma_j \backslash B(t_j/d')$ are disjoint.  
\end{proof}

We finally conclude that the systole gives a lower bound for the uniform depth from below:
\begin{theorem}\label{systolDepth} Let $d$ be the uniform depth of cusps of $X.$ Then, 
\begin{align*} 
    d
    \ge
\operatorname{min}\{
\bi(-2n+s'\bi)^{\frac{1}{4}}
,
\bi(-2n+s'\bi)^{\frac{1}{2}}
\},
\end{align*}
where $s'=e^{\sy/2 }.$

\end{theorem}
\begin{proof}
    Combining Proposition \ref{Trace Sys ineq} with Proposition \ref{Depth Trace Ineq} implies the claim. 
\end{proof}
Direct computation gives the following corollary which will be used later to bound the uniform depth of cusps in terms of $\sy$:

\begin{Corollary}\label{depth Bounds 2}
    If $\sy  \ge 4\ln\bi(5n+(4\pi)^4\bi),$ then
 \begin{align*}  
     d 
     > 
    e^{\sy/16}
    >
     4\pi. 
 \end{align*}

 Additionally, if  $\sy  \ge 4\ln\bi(5n+(8\pi)^4\bi),$ then 
 \begin{align*} 
     d 
     > 
    e^{\sy/16}
    >
     8\pi. 
 \end{align*}
\end{Corollary}

\subsection*{Systole in coverings} 

In this subsection, we study the behavior of the systole under finite \'etale coverings. In Proposition~\ref{positive sys}, we proved that the systole is always positive for non-uniform lattices. In Proposition~\ref{sys-growth}, we show that it is possible to increase the systole by passing to a suitable cover. We also give an example in Proposition~\ref{Picard Modular Surface Ex} to illustrate that this behavior is not limited to normal covers. To establish these results, we first provide a few preliminary lemmas

Parker's generalization of Shimizu's lemma \cite[Page 442]{parker1998volumes} tells us that for a torsion-free lattice, the uniform depth of cusps is at least $2$. Hence, the horoball $\Gamma_i \backslash B_{i}(t_{i}/2)$ is called the canonical horoball around the cusp $c_i$. In particular, if the cusp $c_i$ corresponds to the equivalence class of $q_{\infty}$, then the canonical horoball 
\[
\Gamma_{\infty}\backslash\{(\zeta,v,u)\mid u> t_{\infty}/2\}
\] 
injects into $X$.
 The core of $X$ is the set obtained by removing the canonical horoballs from $X$, and we denote it by $X_{\mathrm{core}}$.
Note that since $X$ is connected and the canonical horoballs are disjoint, the core of $X$ must be nonempty.

\begin{Lemma}\label{Geodesic intersects core} 
Every closed geodesic of $X$ intersects $X_{\mathrm{core}}$.  
\end{Lemma}

\begin{proof}
Suppose not. Then there exists a semisimple element $\gamma \in \Gamma$ corresponding to a closed geodesic that does not intersect $X_{\mathrm{core}}$. Since the canonical horoballs around cusps are disjoint open sets and the geodesic is connected, it must be fully contained in a canonical horoball around a cusp. Pull back the closed geodesic to the (Siegel model of) complex ball $\mathbb{B}^n.$ By conjugating the lattice if necessary, we may assume this canonical horoball is the one around $q_{\infty}.$ 

 Let $c_\gamma$ denote the absolute $c$-entry of $\gamma$ as defined in Definition \ref{well defined rep}(ii). It follows from Lemma \ref{hyperbolic stab} that $q_{\infty}$ has only a parabolic stabilizer in the lattice. Therefore, Lemma \ref{Stab of inifinty} implies that $c_\gamma \neq 0.$ Let $z$ be a point in the canonical horoball around $q_\infty$ such that both $z$ and $\gamma z$ lie in this horoball. By Lemma \ref{shim} and Remark \ref{min c}, we obtain
\[
u(z)u(\gamma z) \leq \frac{4}{|c_\gamma|^2} \leq \frac{t_{\infty}^2}{4}.
\]
This contradicts the fact that both $z$ and $\gamma z$ lie in the canonical horoball $\{(\zeta,v,u)\mid u> t_{\infty}/2\}.$  
\end{proof}

\begin{Lemma}\label{finite dispalacment}
Let $z_0 \in \mathbb{B}^n$ and let $R>0$. Then
$
\{\gamma \in \Gamma \mid d(z_0,\gamma \cdot z_0)\le R\}
$
is finite.
\end{Lemma}
\begin{proof}
Suppose not. Then there exists a sequence $\{\gamma_m\}_{m=1}^{\infty}$ with $d(z_0,\gamma_m \cdot z_0)\le R$. Since the closed ball of radius $R$ is compact, there exists a subsequence, which we again denote by $\{\gamma_m\}_{m=1}^{\infty}$, such that the set $\{\gamma_m \cdot z_0\}_{m=1}^{\infty}$ has an accumulation point $z \in \mathbb{B}^n$. This contradicts the discontinuity of the action of $\Gamma$ (see Remark \ref{discontinous}).
\end{proof}

\begin{Proposition}\label{finite-short-hyperbolics}
For every \(L>0\), there are only finitely many \(\Gamma\)-conjugacy classes of  $\Gamma$ corresponding to closed geodesics of length less than $L.$
\end{Proposition}

\begin{proof}
Fix a fundamental domain $\Sigma$ for $X$ in the universal cover $\mathbb{B}^n.$ Pull back the canonical horoballs to the universal cover and consider the part of $\Sigma$ lying outside them. Denote this set by $\Sigma_{\mathrm{core}}$. Note that $\Sigma_{\mathrm{core}}$ maps to $X_{\mathrm{core}}$, therefore it has to be nonempty. Also, since we removed the neighborhoods of the cusps, this set is bounded. Fix a base point $z_0 \in \Sigma_{\mathrm{core}}$ and let 
$D=\sup\{d(z_0,z)\mid z\in \Sigma_{\mathrm{core}}\}$
be the maximal distance from $z_0$ to a point in $\Sigma_{\mathrm{core}}$.

Let $\gamma\in\Gamma_s$ with $\ell(\gamma)\le L$, and let $A_\gamma\subset\mathbb{B}^n$ denote its axis. The projection of $A_\gamma$ to $X$ is a closed geodesic of length $\ell(\gamma)\le L$, hence it meets $X_{\mathrm{core}}$ by Lemma \ref{Geodesic intersects core}. Therefore there exists $g\in\Gamma$ such that the axis of the conjugate $g\gamma g^{-1}$ meets $\Sigma_{\mathrm{core}}$; in particular we may choose a point
$p \in A_{g\gamma g^{-1}}\cap \Sigma_{\mathrm{core}}.$ Consider the displacement of $z_0$ by $g\gamma g^{-1}.$ Join $z_0$ to $p$, move along the axis by at most $\ell(\gamma)$ (the translation length of $g\gamma g^{-1}$), and then join the endpoint back to $g\gamma g^{-1}z_0$. By the triangle inequality we obtain
\begin{equation*}\label{eq:displacement-bound}
d\!\left(z_0,\;g\gamma g^{-1}z_0\right) 
\le d(z_0,p)+d(p,g\gamma g^{-1}p)+d(g\gamma g^{-1}p, g\gamma g^{-1}z_0) 
\le 2D+L,
\end{equation*}
since $d(g\gamma g^{-1}p, g\gamma g^{-1}z_0)=d(p,z_0)\le D$ and $d(p,g\gamma g^{-1}p)=\ell(\gamma)\le L$.  

Set $R=2D+L$ and define $S_R := \{\gamma \in \Gamma \mid d(z_0,\gamma z_0)\le R\}.$
By Lemma \ref{finite dispalacment}, the set $S_R$ is finite. Hence every $\gamma\in\Gamma$ with $\ell(\gamma)\le L$ is conjugate to some element of the finite set $S_R$. It follows that there are only finitely many conjugacy classes in $\Gamma$ with translation length at most $L.$

\end{proof}

\begin{Remark}\label{inf is min} Fix \(\epsilon>0\). Applying Proposition \ref{finite-short-hyperbolics} with \(L:=\sy+\epsilon\) shows that there are only finitely many lengths of closed geodesics in $X$ not exceeding $L$; that is,
the set \[ \{\ell(\gamma) \mid \gamma \in \Gamma_s,\ \ell(\gamma)\le L\}\]is finite. Therefore, the infimum in the definition of the systole \eqref{sys def} is realized.
\end{Remark}

\begin{definition}\label{cofinal def}
A cofinal normal tower of $X$ is a sequence $\{X_i\}_{i=1}^{\infty}$ of finite \'etale Galois coverings of $X = X_1$, corresponding to a nested sequence of lattices $\{\Gamma_i\}_{i=1}^{\infty}$, where each $\Gamma_i$ is a normal subgroup of $\Gamma_1$, 
\[
\Gamma_{i+1} \subset \Gamma_i \quad \text{and} \quad \bigcap_{i=1}^{\infty} \Gamma_i = \{1\}.
\]

\end{definition}

Since every lattice in $\mathrm{PU}(n,1)$ is finitely generated \cite[Theorem~0.9]{garland1970fundamental}, Malcev's theorem \cite{mal1965faithful} (see \cite[Theorem 7.6.8]{ratcliffe2006foundations}) implies that the lattice $\Gamma$ is residually finite. In particular, there exists a cofinal normal tower for $X$.

\begin{Proposition}\label{sys-growth}
For every $X$, there exists a finite cover $X'$ of $X$ such that $\operatorname{sys}(X')$ is sufficiently large. Moreover, in any cofinal normal tower of coverings $\{X_i\}_{i=1}^{\infty}$ with $X_1 = X$, we have
\[
\operatorname{sys}(X_i) \longrightarrow \infty \quad \text{as } i\to\infty.
\]
\end{Proposition}

\begin{proof}
Fix $L > 0.$ Let $\{X_i\}_{i=1}^{\infty}$ be a cofinal normal tower of $X_1 = X$ with fundamental groups $\{\Gamma_i\}_{i=1}^{\infty}$.  
By Proposition~\ref{finite-short-hyperbolics}, there are only finitely many conjugacy classes in $\Gamma$ corresponding to closed geodesics of length less than $L$. Choose representatives $\gamma_1,\dots,\gamma_m \in \Gamma$ of these classes.  

Since $\bigcap_{i=1}^{\infty}\Gamma_i=\{1\}$ and $\Gamma$ is residually finite, there exists $i_0 > 0$ such that for all $i > i_0$, the subgroup $\Gamma_i$ contains none of the elements $\gamma_1,\dots,\gamma_m$. Because each $\Gamma_i$ is normal in $\Gamma$, it also avoids all of their conjugates. Thus, for such $i$, every nontrivial element of $\Gamma_i$ has a translation length of at least $L$, and hence
$
\operatorname{sys}(X_i) \ge L.
$
Since $L$ was arbitrary, we conclude that $\operatorname{sys}(X_i)\to\infty$ as $i\to\infty$.

The existence of some $X'$ with sufficiently large $\operatorname{sys}(X')$ follows from the fact that $\Gamma$ is residually finite and therefore admits a cofinal normal tower. 
\end{proof}

We emphasize that the phenomenon of the systole becoming arbitrarily large is not limited to coverings from normal cofinal towers. It is enough that the traces of hyperbolic elements become arbitrarily large (see Lemma \ref{length trace}). To illustrate
this, we provide an example of covers that are not normal but for which the systole tends to infinity:

\begin{Proposition}\label{Picard Modular Surface Ex}
Let \(K=\mathbb Q(i)\) with ring of integers \(\mathcal{O}_K=\mathbb Z[i]\), and fix the embedding \(\iota:K\hookrightarrow\C\) with \(\iota(i)=i\). From now on, we regard $\mathcal{O}_K$ as a subring of $\C$
via $\iota.$

Let \(q\equiv 3\pmod 4\) be prime, so \((q)\subset O_K\) is a prime ideal; set \(\mathfrak p=(q)\). We will use the group $\mathrm{U}(Q)$ introduced in the equation \eqref{def: U(Q)}. We will write an element $h\in \mathrm{U}(Q)$ in the form \eqref{MatrixPu}: 
\[
h=\begin{bmatrix}
a & \tau & b\\
\alpha & f & \beta\\
c & \delta & e
\end{bmatrix}
\quad
\text{with } a,b,c,e\in K,\ \ \tau,\alpha,\beta,\delta\in K,\ \ f\in K^{*}.
\]
Consider the group
\(\mathrm{U}(Q)(\mathcal{O}_K):=\mathrm{U}(Q)\cap\mathrm{GL}_3(\mathcal{O}_K)
\)
and define its subgroups by 
\[
\widetilde\Gamma_1(\mathfrak p):=\Big\{\,h\in\mathrm{U}(Q)(\mathcal{O}_K)\,:\,
h\equiv
\begin{bmatrix}
1 & * & *\\
0 & 1 & *\\
0 & 0 & 1
\end{bmatrix}\pmod{\mathfrak p}\Big\}.
\]
Let $\pi:\mathrm{U}(Q)\to \mathrm{PU}(Q)=\mathrm{U}(Q)/\{ \mu I:|\mu|=1\}$ be the natural projection, and set
\[
\Gamma_1(\mathfrak p):=\pi\big(\widetilde\Gamma_1(\mathfrak p)\big)\subset \mathrm{PU}(Q).
\]
Then we define the Picard modular surface at level $\mathfrak p$ as
\[
X_1(\mathfrak p):=\Gamma_1(\mathfrak p)\backslash\mathbb{B}^2,
\]
which is a finite cover of
\(
X(1):=\pi\big(\mathrm{U}(Q)(\mathcal{O}_K)\big)\backslash\mathbb{B}^2.
\)
Then the following hold: 
\begin{enumerate}
\item The covering \(X_1(\mathfrak p)\to X(1)\) is not normal.
\item If \(\gamma\in\Gamma_1(\mathfrak p)\) is hyperbolic, then
\[
|\operatorname{tr}(\gamma)|\;\ge q-3,
\]
where the absolute trace of $\gamma$ is understood as in Definition \ref{well defined rep}.

\item 
\(\operatorname{sys}\big(X_1(\mathfrak p)\big)\to\infty\) as \(q\to\infty\).
\end{enumerate}
\end{Proposition}

\begin{proof}
\begin{enumerate}
    \item Let
\[
g=\begin{bmatrix}1&0&i\\[2pt]0&1&0\\0&0&1\end{bmatrix}\in\widetilde{\Gamma}_1(\mathfrak p),\qquad
h=\begin{bmatrix}0&0&1\\[2pt]0&1&0\\[2pt]1&0&0\end{bmatrix}\in\mathrm{U}(Q)(\mathcal{O}_K).
\]
Then
\[hgh^{-1}=\begin{bmatrix}1&0&0\\[2pt]0&1&0\\i&0&1\end{bmatrix}, \]
whose lower-left entry is a unit modulo \(\mathfrak p\). Hence
\(hgh^{-1}\notin\widetilde{\Gamma}_1(\mathfrak p)\), and therefore its image
under \(\pi\) does not lie in \(\Gamma_1(\mathfrak p)\). This shows that
\(\Gamma_1(\mathfrak p)\) is not normal in \(\pi\big(\mathrm{U}(Q)(\mathcal{O}_K)\big)\).

\item Let $\gamma\in\Gamma_1(\mathfrak p)$ be hyperbolic, and choose a lift
$\widetilde\gamma\in\widetilde\Gamma_1(\mathfrak p)$ of $\gamma.$ By the definition of $\widetilde\Gamma_1(\mathfrak p)$ we have
$\operatorname{tr}(\widetilde\gamma)\equiv 3 \pmod{\mathfrak p},$
so
$t:=\operatorname{tr}(\widetilde\gamma)-3\in\mathfrak p.$ Since $\gamma$ is hyperbolic, $\widetilde\gamma$ is not unipotent and thus $t\neq0$.
Because $\mathfrak p=(q)$ with $q\equiv 3\pmod 4$ is prime in $\mathcal O_K=\mathbb Z[i]$,
we may write $t=q\,w$ with $w\in\mathbb Z[i]\setminus\{0\}$. Taking complex absolute values yields
\[
\bigl|\operatorname{tr}(\widetilde\gamma)-3\bigr|=|t|=q\,|w|\ge q,
\]
since $|w|\ge1$ for every nonzero $w\in\mathbb Z[i]$. This implies that:
$\bigl|\operatorname{tr}(\gamma)\bigr|=\bigl|\operatorname{tr}(\widetilde\gamma)\bigr|\ge q-3.$

\item This follows from the second part and Lemma \ref{length trace}.
\end{enumerate}

\end{proof}

\section{Thick-thin decomposition}\label{thin-thick decomposition}
In this section, we introduce a version of the thick-thin decomposition relative to the systole, which differs from Margulis' decomposition as it depends on the lattice $\Gamma$. The main goal of this section is to prove Theorem \ref{thick}, which states that the thin part of $X$ contains no subvariety 
(by a subvariety of $X$ we mean the intersection of a closed, irreducible, 
positive-dimensional algebraic subvariety of the projective variety $\Xb$ 
with $X$, such that the intersection is nonempty).

Let $c_i$ be a cusp of $X$ with unipotent stabilizer $\Gamma_{i}.$
Fix $\epsilon>0.$
Consider the set
$$
\tilde{U}_{i, \epsilon}= \{x\in \mathbb{B}^n| \exists g\in \Gamma_{i}, d(x,g\cdot x)< \epsilon\}.
$$
We define the $\epsilon$-thin neighborhood around the cusp $c_i$ as the set $U_{i, \epsilon}:=\Gamma_i \backslash\tilde{U}_{i, \epsilon}.$ Also, we fix $\rho=\sy/2$ and define the thin part of $X$ as the union of all $\rho$-thin neighborhood around cusps of $X:$
$$X_{\mathrm{thin}}:=\dis\cup_{i=1}^{k}{U_{i,\rho}},$$
where $k$ is the number of cusps. The following Proposition shows that $X_{\mathrm{thin}}$ is actually the disjoint union of the $\rho$-thin neighborhood around cusps:

\begin{Proposition}\label{disjoint Thin}
    If $\epsilon<\sy/2$, then $U_{i, \epsilon} \cap U_{j, \epsilon}=\varnothing$ for $i\neq j.$
\end{Proposition}
\begin{proof}
     For the sake of the contradiction assume that $x\in U_{i, \epsilon} \cap U_{j, \epsilon}.$ This means that there exist $\gamma_1\in \Gamma_{i}$ and $\gamma_2\in \Gamma_{j}$ such that 
     $d(\tilde{x},\gamma_1 \cdot \tilde{x})<\epsilon$ and $d(\tilde{x},\gamma_2 \cdot \tilde{x})<\epsilon,$ where $\tilde{x}\in \mathbb{B}^n$ is a fiber of $x.$ This in particular implies that $d(\tilde{x}, \gamma_1^{-1} \cdot \tilde{x} )< \epsilon.$ 
     
     Since $\mathrm{PU}(Q)$ acts doubly transitively on the boundary, we can, if necessary, change coordinates (i.e., by conjugating the lattice) to identify $c_i$ and $c_j$ with the equivalence classes of points $q_\infty$ and $q_0$ on the boundary $\partial \Sn$. Note that the systole is invariant under this change of coordinates.

     We represent $\gamma_1$ and $\gamma_2$ by the  matrices $g_{\infty},g_{0} \in \mathrm{PU}(Q)$ respectively, where
\begin{align*}
    g_{\infty}=
\begin{bmatrix}
    1&-\tau^*&-(|\tau|+it)/2\\
    0&I_{n-1}&\tau\\
    0&0&1\\
\end{bmatrix},
\ \ \ \
    g_{0}=
\begin{bmatrix}
    1&0&0\\
    \sigma&I_{n-1}&0\\
    -(|\tau|+is)/2&-\sigma^*&1\\
\end{bmatrix}.
\end{align*}
Note that $\gamma_1^{-1}=(-\tau, -t)$ corresponds to $g^{-1}_{\infty}.$ We can write:
\begin{align*}
    \bi|\operatorname{tr}(g_{\infty}g_{0})\bi|+\bi|\operatorname{tr}(g^{-1}_{\infty}g_{0})\bi|
    &\ge
    \Big |\operatorname{tr}\bi( (g_{\infty}+g_{\infty}^{-1})g_{0} \bi) \Big| \\
    &= \Big | \operatorname{tr}\Big(
    \begin{bmatrix}
       2&0&-|\tau|^2 \\
       0&2 I_{n-1}&0  \\
       0&0&2  \\
    \end{bmatrix}
    \begin{bmatrix}
       1&0&0 \\
       \sigma&I_{n-1}&0  \\
       (-|\sigma|^2+is)/2&-\sigma^*&1  \\
    \end{bmatrix} \Big) \Big| \\
    &= \bi|2(n+1)+\frac{1}{2}|\tau|^2(|\sigma|^2-is) \bi|\\
    &\ge 2(n+1)+\frac{1}{2}|\tau|^2 |\sigma|^2.
\end{align*}
Hence, either $ | \operatorname{tr}( g_{\infty} g_{0})|\ge n+1+\frac{1}{4}|\tau|^2 |\sigma|^2 $ or $ | \operatorname{tr}( g^{-1}_{\infty} g_{0})|\ge n+1+\frac{1}{4}|\tau|^2 |\sigma|^2$ and therefore either $\gamma_1 \gamma_2 $ or $\gamma_1^{-1} \gamma_2$ must be hyperbolic. But this implies that either $d(\tilde{x}, \gamma_1 \gamma_2  \cdot \tilde{x}) \ge \sy $ or $d(\tilde{x}, \gamma_{1}^{-1} \gamma_2  \cdot \tilde{x})\ge \sy ,$ which is a contradiction because $d(\tilde{x},\gamma_2\cdot  \tilde{x})< \sy/2,$  $d(\tilde{x}, \gamma_1 \cdot \tilde{x}) < \sy/2,$ and $d(\tilde{x}, \gamma_1^{-1} \cdot \tilde{x})< \sy/2.$ 
\end{proof}
We define the thick part of $X$ as the complement of the thin part: 
$$X_{\mathrm{thick}}:=X \setminus \dis\cup_{i=1}^{k}{U_{i,\rho}}.$$
Since every point in a thin part of $X$ has a displacement less than $\sy/2,$ the following Proposition tells us that $X_{\mathrm{thick}}\neq \varnothing.$
\begin{Proposition} \label{radofThick}
    There exists $x\in X$ such that 
    $$
    \operatorname{inj_{x}(X)}
    \ge
    \sy /2.
    $$
\end{Proposition}
\begin{proof}
     Note that if $\gamma\in \Gamma$ is not unipotent, then it is semi-simple and for such $\gamma$ and every $x\in \mathbb{B}^n,$ we have $d(x,\gamma\cdot x)\ge \sy.$

   Now, assume for the sake of contradiction that  $\operatorname{inj}_{x}(X)
    <
    \sy/2$ for all $x\in X.$
    Therefore, for all $x$ there is a unipotent element $\gamma\in \Gamma$ such that $d(x,\gamma\cdot x)<
    \sy/2.$ This means that the thin part of $X$ covers all $X.$ However, this is not possible because the thin part of $X$ is a disjoint union of open sets by Proposition \ref{disjoint Thin} but $X$ is connected.    
   
\end{proof}

Now we show that the monodromy of the $\rho$-thin part of $X$ around each cusp is in the stabilizer of that cusp: 

\begin{Lemma} \label{monomdromy} Suppose that $\epsilon<\sy/2 
.$
Let $U'_i$ be a connected component of $U_{i, \epsilon}$ and $\iota:U'_i \to X$ be the identity map. Then, $\iota_*(\pi_{1}(U_i'))$ is a subgroup of $ \Gamma_{i}.$
\end{Lemma}
\begin{proof} As $\epsilon<\sy/2,$ Proposition \ref{disjoint Thin} implies that $U_{i,\epsilon}$s are disjoint. Fix $x\in U'_{i}$ and let $\gamma:[0,1] \to X$ be a  loop at $x$ which is a representative of a class in $\iota_*(\pi_{1}(U'_i,x)).$ Let $\tilde{x}$ be a lift of $x$ to the universal cover $\mathbb{B}^n.$ As $\gamma$ is fully contained in $U'_{i,\epsilon},$ we can lift it to a path $\tilde{\gamma}:[0,1]\to \tilde{U}_{i,\epsilon}$ which starts at $\tilde{x}.$ Therefore, $\tilde{y}:=\gamma \cdot \tilde{x}=\tilde{\gamma}(1)$ is in $\tilde{U}_{i,\epsilon}.$ Let $\gamma' \in \Gamma_i$ such that 
$d(\tilde{x}, \gamma'\cdot \tilde{x})<\epsilon.$ By homogeneity, we have
$$
d(\tilde{y}, \gamma \gamma' \gamma^{-1} \cdot \tilde{y})
=
d(\gamma \cdot \tilde{x}, \gamma \gamma' \gamma^{-1} \gamma \cdot \tilde{x})
=
d(\tilde{x}, \gamma'\cdot \tilde{x})
<
\epsilon
.
$$
Since $\gamma \gamma' \gamma^{-1}$ fixes $\gamma(q_i),$ the previous inequality tells us that $\tilde{y}$ is in the $\epsilon-$thin neighborhood around $\gamma(q_i).$ On the other hand we know $\tilde{y} \in U_{i,\epsilon}.$ Since the thin neighborhoods around cusps are disjoint therefore $\gamma(q_i)=q_i,$ i.e., $\gamma \in \Gamma_i,$ as desired.
\end{proof}

Finally, we show that every subvariety of $X$ intersects with $X_{\mathrm{thick}}$, that is, every subvariety of $X$ contains a point whose injectivity radius in $X$ is larger than $\sy/2:$

\begin{theorem}\label{thick}
     Every subvariety of $\Xb$ either intersects with $X_{\mathrm{thick}}$ or fully contained in the boundary $D,$ where $\Xb=X\cup D.$
\end{theorem}
           \begin{proof}
   For the sake of contradiction, assume that there exists a connected subvariety $V$ fully contained in $X_{\mathrm{thin}}$. By Proposition \ref{disjoint Thin}, $V$ must be contained within a thin neighborhood of a cusp of $X$, say \(c_i\). With a change of coordinates if necessary (i.e., by conjugating the lattice \(\Gamma\)), we can assume that \(c_i\) is identified with the equivalence class of \(q_\infty\). Note that the systole of \(X\) is invariant under this change of coordinates, and therefore so is \(X_{\mathrm{thin}}\).

 Consider the function $-u$ which is a plurisubharmonic function on the Siegel domain $\Sn$ and invariant under the action of stabilizer $\Gamma_{\infty}$(see \cite[\S 2]{bakker2018kodaira}). Therefore, it follows from Lemma \ref{monomdromy} that $-u$ is a well-defined function on every component of the thin part around $q_{\infty}.$ Hence, $-u$ is a well-defined plurisubharmonic function on $V.$ Notice that if a plurisubharmonic function achieves its maximum on a closed connected variety, it has to be constant(see \cite[page 272]{gunning2022analytic}). Since $V$ is compact, $-u$ must be constant on $V.$ However, it is not possible because the K\"ahler form on $X$ is induced by $-2i \partial \bar{\partial} \log(u)$(see \cite[Lemma 2.1] {bakker2018kodaira}) and if $-\log(u)$ were constant, the induced K\"ahler volume of $V$ would be zero.    
\end{proof}
 \section{Volume estimate of subvarieties}
In this section, we prove Theorem \ref{Hyperbolic Vol Int}, and Theorem \ref{LargeIntersection Int}. 
We first state Hwang and To's theorem in the following way:

\begin{theorem}(\cite[Theorem 1.1]{hwang2002volumes})\label{Hwanginequilty} Take $x \in X$ with injectivity radius $r=\operatorname{inj}_{x}(X).$ Let $B(x,r)$ be the Bergman ball of radius $r$ centered at $x.$ Suppose $V$ is an $m$-dimensional subvariety of $X$ passing through $x.$
Then, the following inequality holds:
\begin{align}  
\operatorname{vol_{X}(V\cap B(x, r) )}
&\ge
\frac{(4\pi)^m}{m!}\sinh^{2m}(r)\cdot \operatorname{mult}_x(V).
\end{align} 
\end{theorem}
Hwang and To generalized the above-mentioned theorem for a general Hermitian symmetric domain in \cite{hwang2000seshadri,hwang2002volumes}.

In the compact case, Theorem \ref{Hwanginequilty} gives the lower bound on the induced K\"ahler volume of subvarieties in terms of the injectivity radius of $X,$ however, in the case that $X$ is not compact, the injectivity radius of $X$ goes to zero as we get closer to the cusps. So we use the systole, the length of a shortest closed geodesic in $X,$ as a geometric invariant of $X$ to uniformly bound the volumes of all subvarieties of $X.$  For a compact ball quotient, the systole is twice the injectivity radius. However, for a non-compact $X$ the systole is still not zero and can be estimated by the trace of the hyperbolic elements in a representation of $\Gamma.$

 Theorem \ref{thick} tells us that every subvariety of $X$ has a point with injectivity radius as large as $\sy/2$. Hence, we will get the following theorem:

\begin{theorem} \label{Bound Hyperbolic Volume}
    Let $V$ be an $m$-dimensional subvariety of $\Xb$ which is not contained in $D.$ Then,
\begin{align} 
          \operatorname{vol}_{X}(V)
     &\ge
      \frac{ (4 \pi)^m}{m!}\sinh^{2m}\bi(\sy /2\bi).        
\end{align}    
\end{theorem}
\begin{proof}
Theorem \ref{thick} implies that $V\cap X_{\mathrm{thick}}\neq \varnothing.$
This means there always exists a point $x \in V$ such that $\inj \ge \sy /2.$ Now, Theorem \ref{Hwanginequilty} gives:
    $$
     \operatorname{vol}_{X}(V)
     \ge
      \frac{ (4\pi)^m}{m!}\sinh^{2m}\bi(\sy /2\bi).
    $$
\end{proof}

\begin{Corollary}\label{Deg of log} With the same notation as Theorem \ref{Bound Hyperbolic Volume}, we have that
     \begin{align*}\nonumber
     (K_{\Xb}+D) ^{m} \cdot V
     &\ge
         (n+1)^m\sinh^{2m}\bi(\sy /2\bi).
    \end{align*}
\end{Corollary}
\begin{proof} Theorem \ref{Bound Hyperbolic Volume} together with \eqref{ChernForm} gives:
    \begin{align*}
     (K_{\Xb}+D) ^{m} \cdot V
     &= 
     \bi(\frac{n+1}{4\pi }\bi)^m m!
     \operatorname{vol}_{X}(V) \mbox{\ (by \eqref{ChernForm})}   \\ \nonumber
     &\ge
     (n+1)^m\sinh^{2m}\bi(\sy /2\bi) \mbox{\ (by Theorem \ref{Bound Hyperbolic Volume})} 
    \end{align*}
\end{proof}
We recall Bakker and Tsimerman's theorem which tells us that 
the uniform depth of cusps of $X$ bounds the intersection numbers of $K_{\Xb}$ with subvarieties of $\Xb$ which is not contained in $D.$
\begin{theorem}(\cite[Corollary 3.8]{bakker2018kodaira})\label{Tsi} Suppose $d$ is the uniform depth of cusps. Then,
    $$K_{\Xb}+(1-\lambda) D$$
is ample for $\lambda \in (0,d(n+1)/4\pi)$.
\end{theorem}

\begin{remark}
    There is a typo in the statement of this corollary in the paper by Bakker-Tsimerman. The correct upper bound should be $d(n+1)/4\pi,$ as is clear from their proof and has been confirmed by the authors.
\end{remark}
Consider the decomposition of the boundary divisor $D$ to the connected components $D=\cup_{m=1}^{k}D_m.$ Due to \cite{MOK}, we know that each $D_{m}$ is an abelian variety with ample conormal bundle $O_{D_m}(-D_m).$

\begin{Lemma}\label{two component positivity}
If $d>4\pi$, then for $i\neq j$ the line bundle
$L_{i,j}:=K_{\Xb}-D_i-D_j$ is big and nef.
\end{Lemma}

\begin{proof}
since $d>4\pi,$ by Theorem \ref{Tsi} the divisor $K_{\Xb}-\sum_{m=1}^k D_m$ is ample. Writing
\[
L_{i,j}=\Big(K_{\Xb}-\sum_{m=1}^k D_m\Big)+\sum_{m\ne i,j}D_m,
\]
we see that $L_{i,j}$ is the sum of an ample divisor and effective divisors, hence $L_{i,j}$ is big.

For nefness let $C\subset\Xb$ be any irreducible curve.  We consider two separate cases:
\begin{enumerate}
\item The curve $C$ is contained in $D_r$ for some $r\neq i,j:$ 

Since the boundary components are pairwise disjoint, for $r \neq i,j$ the restriction $D_m|_{D_r}$ is trivial for all $m \neq r$. Adjunction gives $(K_{\Xb}+D_r)|_{D_r}\cong K_{D_r}.$
We know that each $D_r$ is an abelian variety, so $K_{D_r}\cong\mathcal O_{D_r}$. This yields that $(K_{\Xb})_{|D_r}\cong\mathcal O_{D_r}(-D_r).$
Therefore, for $r\neq i,j$,
\[
L_{i,j}|_{D_r}=(K_{\Xb}-D_i-D_j)_{|D_r}\cong (K_{\Xb})_{|D_r}\cong\mathcal O_{D_r}(-D_r).
\]
Since the conormal bundle $\mathcal O_{D_r}(-D_r)$ is ample on $D_r$, its degree on any curve $C\subset D_r$ is positive. Thus $L_{i,j}\cdot C>0$ in this case.

\medskip\item The curve $C$ is not contained in any boundary component $D_m$ except for $m=i$ or $m=j:$
Then for every $m\neq i,j,$ we know that $D_m\cdot C\ge0$ (intersection with an effective divisor is non-negative when the curve is not contained in that divisor), and because $K_{\Xb}-D$ is ample we have \((K_{\Xb}-D)\cdot C>0\). Hence
\[
L_{i,j}\cdot C = (K_{\Xb}-D)\cdot C + \sum_{m\ne i,j} D_m\cdot C > 0.
\]
\end{enumerate}
In both cases $L_{i,j}\cdot C>0$ for every irreducible curve $C$, so $L_{i,j}$ is nef.
\end{proof}

Now, putting together Theorem \ref{Bound Hyperbolic Volume}, Theorem \ref{Tsi} and what we proved for the uniform depth of cusps, Theorem \ref{systolDepth}, yields a lower bound for the degree of $K_{\Xb}$ on $V$ in terms of $\sy:$ 
\begin{theorem} \label{LargeIntersection} Let $V$ be an $m$-dimensional subvariety of $\Xb$ which is not fully contained in $D.$ If $\sy  \ge 4\ln\bi(5n+(4\pi)^4\bi),$ then
\begin{align} \label{deg}
    K_{\Xb}^m \cdot V 
    &>
    (\frac{n}{4\pi})^m \cdot e^{m\sy/16 }.
\end{align}    
\end{theorem}
\begin{proof}

We deal separately with the following two cases:
\begin{enumerate}
    \item $V\cap D= \varnothing:$ In this case we have $K_{\Xb} ^{m} \cdot V
     =
     (K_{\Xb}+D) ^{m} \cdot V$  
 and from Corollary \ref{Deg of log} we get that 
    \begin{align} \nonumber
    K_{\Xb} ^{m} \cdot V
      & \ge \label{Vol subvar out D}
      (n+1)^m\sinh^{2m}\bi(\sy /2\bi)\\
      & >
      (\frac{n}{4\pi})^m e^{m\sy} \mbox{\ \ (by the bound on the systole).}
    \end{align}
    
    \item $V\cap D\neq \varnothing:$ Since $\sy  \ge 4\ln\bi(5n+(4\pi)^4\bi)),$ Corollary  \ref{depth Bounds 2} gives that the uniform depth of cusps is at least $4\pi$ and it follows from Bakker-Tsimerman's theorem (Theorem \ref{Tsi})  that $K_{\Xb}$ is ample. In particular, this implies that $K_{\Xb|D}$ is ample. On the other hand, we know that the conormal bundle $-D_{|D}$ is ample. Therefore, for every $i>1,$ we have
    \begin{align} \label{Negative middle term}
        K_{\Xb}^{m-i}\cdot (-D)^i \cdot V
        =
        - (K_{\Xb|D})^{m-i} \cdot (-D_{|D})^{i-1} \cdot V_{|D} <0. 
    \end{align}

 By Bakker-Tsimerman's  theorem, Theorem \ref{Tsi}, we get that  
$$\Big(K_{\Xb}-\bi((n+1)d/4\pi-1\bi)D\Big)^m \cdot V \ge 0.$$
Expanding this and combining with \eqref{Negative middle term} gives:
\begin{align} \nonumber
   K_{\Xb}^m \cdot V
   &\ge 
   \bi((n+1)d/4\pi-1\bi)^m \cdot -(-D)^m \cdot V \\ \nonumber
   &\ge
   \bi((n+1)d/4\pi-1\bi)^m  \mbox{(by ampleness of $-D_{|D}$)}\\ \nonumber
 \nonumber
&>
\Big(nd/4\pi
 \Big)^m  \mbox{ (because $d> 4\pi$)} \\
&\ge \label{Vol subvar in D}
\bi(n/4\pi\bi)^m
 \cdot
e^{m\sy/16 } \mbox{ (by Corollary  \ref{depth Bounds 2})}.
\end{align}

Combining \eqref{Vol subvar out D} and \eqref{Vol subvar in D} gives that for all $V$ not contained in $D$ we have:
\begin{align*}
    K_{\Xb}^m \cdot V > (\frac{n}{4\pi})^m \cdot e^{m\sy/16 }.   
\end{align*}
\end{enumerate}  
\end{proof}
 The volume of a line bundle $L$ on an $m$-dimensional projective variety $V$ is defined as the non-negative real number 
$$\operatorname{vol}_{V}(L):=\limsup\limits_{b\rightarrow \infty}\dfrac{h^0(V, bL)}{b^m/m!},$$
which measures the positivity of $L$ from the point of view of birational geometry. If $L$ is a nef line bundle on $V$, then $\operatorname{vol}_{V}(L)= L^{ n}.$ Let $V'$ be a smooth variety birational to $V$ with a canonical bundle $K_{V'}.$ The canonical volume of the subvariety $V$ is

$$\widetilde{\operatorname{vol}}_{V}:=\limsup\limits_{b\rightarrow \infty}\dfrac{h^0(V', bK_{V'})}{b^m/m!},$$
which does not depend on the choice of $V'.$

To prove the bound on the canonical volume for a subvariety which does not intersect the boundary $D,$ we will use the following lemma inspired by \cite[Proposition 3.2]{brunebarbe2020increasing}.

We refer the reader to \cite{goldberg1967holomorphic} for basics on different notions of curvatures and here we will use the facts that the holomorphic sectional curvature of the Bergman metric is $-1$ and the holomorphic bisectional curvature of this metric is bounded above by $-\frac{1}{2}.$
\begin{Lemma} \label{Richi} Let $V$ be an $m-$dimensional subvariety of $\Xb$ which does not intersect with $D.$ Let $\omega$ be the K\"ahler form induced on $V$ from the Bergman metric. Then on $V$ the following inequality holds
\begin{align*}
    \operatorname{Ricci}_\omega \le -\frac{m+1}{2}\omega,
\end{align*}
where $\operatorname{Ricci}_\omega$ is the Ricci curvature of $\omega. $
    
\end{Lemma}
\begin{proof} Since both sides of the inequality are bilinear, it is enough to check the inequality only for unit vectors. Let $x\in V$ and $v\in T_{x}V$ be a unit vector, i.e., $\omega(v,v)=1.$ Take an orthonormal basis
$(e_1,e_2 \dots , e_m)$ of $T_{x}V$ such that $e_1 = v.$ We will denote the holomorphic bisectional curvature of the Bergman metric at $u_1,u_2\in T_{x}V$ by $H(u_1,u_2).$ Since the holomorphic bisectional curvature and holomorphic sectional curvature only decrease on subvarieties, we have that $H(e_i,v)\le -\frac{1}{2}$ for $i\in\{2,...,m\}$ and $H(e_1,e_1)\le -1$. Now we can write:
\[\operatorname{Ricci}_\omega(v,v)= \sum_{i=1}^{n}H(e_i,v)=H(e_1,e_1) +\sum_{i=2}^{n}H(e_i,v)\le -\frac{m+1}{2}.\]
    
\end{proof}

\begin{theorem} \label{algebriac Volume} Let $V$ be an $m$-dimensional subvariety of $\Xb$ which is not contained in $D.$ If $\operatorname{sys(X)}\ge 4\ln\bi(5n+(8\pi)^4\bi),$ then
\begin{align} 
  \widetilde{\operatorname{vol}}_{V}
    >
     (\frac{m}{4\pi})^m e^{m\sy/16}.
\end{align}    
\end{theorem}    
\begin{proof} Let $\mu: V'\longrightarrow V$ be a desingularization such that the set-theoretic preimage of the boundary divisor $D'$ is a normal crossing divisor. We consider two cases:
\begin{enumerate}
    \item $V \cap D =\varnothing:$ Since on $V$ we have that $\operatorname{Ricci}_\omega=-2\pi i \cdot c_1(K_V),$ we can apply Lemma \ref{Richi} together with \eqref{ChernForm} to get that $K_{V'}-\frac{m+1}{n+1}\mu^*(K_{\Xb})$ is nef and in particular pseudo-effective on $V'.$ Because the 
 volume does not decrease in a pseudo-effective direction, we get:
   \begin{align}\label{cano vol ineq 1}
    \widetilde{\operatorname{vol}}_{V}
    =
    \operatorname{vol}_{V'}(K_{V'}) 
    \ge
    (\frac{m+1}{n+1})^m\operatorname{vol}_{V'}(\mu^{*}(K_{\Xb})). 
     \end{align}
    Since $\sy \ge 4\ln\!\big(5n+(4\pi)^4\big)$, Corollary~\ref{depth Bounds 2} implies that the uniform depth of cusps is at least $4\pi$. Therefore, by Theorem~\ref{Tsi}, $K_{\Xb}$ is ample, and hence $\mu^*(K_{\Xb})$ is big and nef. Now we can write: 
\begin{align*}
\operatorname{vol}_{V'}(\mu^*(K_{\Xb}))
= (\mu^*(K_{\Xb}))^m 
&=K_{\Xb}^m\cdot V\\
&\ge (n+1)^m\sinh^{2m}\bi(\sy /2\bi)
\mbox{ \ (by \eqref{Vol subvar out D})}.\end{align*}

We can conclude the desired inequality for this case by combining this inequality with \eqref{cano vol ineq 1}. 

\item  $V \cap D \neq \varnothing:$ Since $\sy  \ge 4\ln\bi(5n+(8\pi)^4\bi)),$ Corollary \ref{depth Bounds 2} gives that the uniform depth of cusps is at least $8\pi.$ By \cite[Theorem A]{memarian2022positivity} we get that twisted log-cotangent bundle $\Omega_{V'}^1(\operatorname{log}(D))\la -rD' \ra$ is big and nef for every $r\in(0,d/4\pi).$ Taking the determinant gives that $K_{V'} + (1 - mr)D'$ is big and nef for every $r \in (0, d/4\pi)$.  
Since $d > 4\pi$, we can plug in $r = 1/m$ and deduce that $K_{V'}$ is big and nef. Hence,
\[
\widetilde{\operatorname{vol}}_V = \operatorname{vol}_{V'}(K_{V'}) = K_{V'}^m.
\]

On the other hand, we know that the bundle $-D'_{|D'}$ is big and nef. Therefore, for every $i>1,$ we have
    \begin{align} \label{Negative middle term 2}
        K_{V'}^{m-i}\cdot (-D')^i \cdot V'
        =
        - (K_{V'|D})^{m-i} \cdot (-D'_{|D'})^{i-1} \cdot V_{|D'} \le 0. 
    \end{align}

Let $r'$ be a rational number between $d/8\pi$ and $d/4\pi.$ Since $K_{V'}+(1-m r')D'$ is big and nef we have 
$\Big(K_{V'}+(1-m r')D'\Big) ^m \ge 0.$ Expanding this and using $\eqref{Negative middle term 2}$ gives that 
\begin{align*}
  K_{V'}^m
&\ge \bi( mr' -1 \bi )^m(-D'_{|D'})^m \\
&\ge \bi( mr' -1 \bi )^m \mbox{ \ (because $-D'_{|D'}$ is big and nef) \ }\\
&>
(md/8\pi)^m  \mbox{ \ (because $r'> d/8\pi> 1$)}\\
&\ge 
(\frac{m}{4\pi})^m e^{m\sy/16}
 \mbox{ (by Corollary \ref{depth Bounds 2})}.
\end{align*}

\end{enumerate}
\end{proof}

\section{Effective global generation and very ampleness}
In this section, we prove Corollary \ref{Very amplness Int}, Corollary \ref{sepration of jets and seshadri cons int} based on the bound we found for $\operatorname{deg}_{\Xb}(V)$ in Theorem \ref{LargeIntersection}. First, we analyze the problem on the boundary divisor $D.$ 

\subsection{Base-point freeness and very ampleness on $D$}  
In this subsection, we prove that if the uniform depth of cusps is sufficiently large, then $2K_{\Xb}$ does not have a base point on $D,$ and moreover $3K_{\Xb}$ can separate any two points, and any tangent direction on $D.$ We first prove that the restricted bundles on the boundary satisfy these properties. Consider the decomposition of the boundary divisor $D$ to the connected components $D=\cup_{i=1}^{k}D_i.$ Due to \cite{MOK}, we know that each $D_{i}$ is an abelian variety with ample conormal bundle $O_{D_i}(-D_i).$ 

\begin{Lemma}\label{ResBun}The line bundle
$2K_{\Xb|D_i}$ is base-point free and $3K_{\Xb|D_i}$ is very ample for every $i.$ 
\end{Lemma}
\begin{proof}
The adjunction formula gives that $ K_{\Xb|D_i}\cong -D_{i|D_i}.$ As the conormal bundle is ample and $D_i$ is an abelian variety, $-2 D_{i|D_i}$ is base-point free and $-3 D_{i|D_i}$ is very ample (see \cite{ohbuchi1987some}). 

\end{proof}
In the next two lemmas, we see how we can lift the sections from the restricted bundle to $\Xb.$
The base locus of a line bundle $L$ on $\Xb$ will be denoted by $\operatorname{Bs}(L).$
\begin{Lemma} \label{bndry} Suppose that the uniform depth of cusps is larger than $4\pi,$ Then, the following hold:
\begin{enumerate}
    \item $\operatorname{Bs}\bi (2K_{\Xb} \bi )\cap D= \varnothing$ 
    \item For any two points on different components of $D,$ there exists a global section of $2K_{\Xb}$ which separates them. 
\end{enumerate} 
\end{Lemma}    
\begin{proof}
Let $L$ be $2K_{\Xb}.$
\begin{enumerate}
    \item By Lemma \ref{ResBun}, $L_{|D}$ is base-point free and therefore it is enough to show that we can lift the global sections from $D$ to $\Xb$, that is,
    $H^0(\Xb, L) \longrightarrow H^0(D, L_{|D})$ is surjective.
    Consider the following exact sequence on $\Xb:$
     $$0 \longrightarrow L-D 
    \longrightarrow L
    \longrightarrow L_{|D}
    \longrightarrow 0.
    $$
    Writing the long exact sequence we can see that it is sufficient to show $H^1(\Xb, L-D)=0.$ 
   As $L-D= K_{\Xb} +(K_{\Xb}-D),$ if the uniform depth is sufficiently large, then by Theorem \ref{Tsi} $K_{\Xb}-D$ is ample. Therefore, the vanishing of $H^1(\Xb, L-D)$ follows from Kodaira's vanishing theorem. 
   
\item Suppose that we want to separate $x\in D_i$ and $y\in D_j$ with $i\neq j.$ It is sufficient to find a global section of $L_i:=L-D_i$ which does not vanish at $y.$ We can argue similar to the first part. Concretely, since $D_i$ and $D_j$ are disjoint, the line bundle $L_{i|D_j}$ is isomorphic to the line bundle $L_{|D_j},$ which we know is base-point free by Lemma \ref{ResBun}. Therefore, it is enough to show that we can lift the global sections on $D_j$ to $\Xb$, that is,
    $H^0(\Xb, L_i) \longrightarrow H^0(D, L_{i|D_j})$ is surjective. Consider the exact sequence 
   $$0 \longrightarrow L_i-D_j 
    \longrightarrow L_i
    \longrightarrow L_{i|D_j}
    \longrightarrow 0.
    $$
 Writing the long exact sequence we can see that it is sufficient to show $H^1(\Xb, L_i-D_j)=0.$ By Lemma \ref{two component positivity}, we get that $K_{\Xb}-D_i-D_j$ is big and nef. Since $L_i-D_j=K_{\Xb}+(K_{\Xb}-D_i-D_j)$, the Kawamata-Viehweg vanishing theorem implies that $H^1(\Xb, L_i-D_j)=0.$ 
 
\end{enumerate}
\end{proof}

\begin{Lemma} \label{Very Ampleness on Boundray}
    If the uniform depth of cusps is larger than $2\pi,$ then $3K_{\Xb}$ can 
    separate any two points on a connected component of $D,$ and at each point of 
$D$ it can separate any two tangent directions.

\end{Lemma}
\begin{proof}
    By Lemma \ref{bndry} and Lemma \ref{ResBun}, it is enough to show that we can lift the sections from the boundary, i.e.,
    $$H^{0}(\Xb, 3K_{\Xb}) \longrightarrow 
        H^{0}(D, 3K_{\Xb|D}) \longrightarrow
        0.
    $$
    Hence, it is enough to show that $H^1(\Xb, 3K_{\Xb}-D)=0.$ Since $d>2\pi,$ it follows from Theorem \ref{Tsi} that $2K_{\Xb}-D$ is ample. Therefore, by Kodaira's vanishing theorem we get that $H^1(\Xb, 3K_{\Xb}-D)=0.$ 
\end{proof}
\subsection{Global generation and very ampleness on $\Xb$}
In this subsection, we see how we can conclude effective global generation and effective very ampleness results by using Theorem \ref{LargeIntersection}. We first recall the famous theorem of Angehrn and Siu on pointwise base-point freeness:   
\begin{theorem} \cite[Theorem
0.1]{angehrn1995effective} \label{Siu Angehrn} Let $Y$ be a smooth projective variety of dimension $n,$ and let $L$ be an ample line bundle on $Y.$ Fix a point $y\in Y,$ and assume that 
\begin{align}
    L^{m}\cdot V > \bi(\frac{n(n+1)}{2}\bi)^m
\end{align}
for every subvariety $V$ of dimension $m$ passing through $y.$
Then, $K_{Y}+L$ has a section that does not vanish at $y.$   
\end{theorem}
Combining Angehrn and Siu's result with our Theorem \ref{LargeIntersection} gives that if $\sy$ is sufficiently large relative to $n,$ then $2K_{\Xb}$ is globally generated: 
\begin{theorem} \label{Freeness}
If $\operatorname{sys(X)}\ge 20\ln\bi(5n+(4\pi)^4\bi),$ then $2K_{\Xb}$ is globally generated.
 \end{theorem}
\begin{proof}
Using Corollary  \ref{depth Bounds 2} we get that $d>4\pi.$ Therefore, by Lemma \ref{bndry}, $2K_{\Xb}$ does not have any base point on $D.$ On the other hand, Theorem \ref{LargeIntersection} implies that for every $m$-dimensional subvariety $V \subset \Xb$ which is not contained in $D,$ we have
\begin{align*}
     K_{\Xb}^{m}\cdot V 
     &\ge 
      (\frac{n}{4\pi})^m \cdot e^{m\sy/16 } \\
     &\ge 
     n^m  \bi(5 n+(4\pi)^4\bi)^m   \mbox{(by the bound on $\sy$)} \\
     &> 
     \bi(\frac{n(n+1)}{2}\bi)^m . 
\end{align*}
Therefore, Theorem \ref{Siu Angehrn} implies that for every point $x\in \Xb \setminus D,$ there is a section of $2K_{\Xb}$ which does not vanish at $x.$ Hence, $2K_{\Xb}$ is globally generated.     
 \end{proof}
 Now, we prove a proposition which will be used to show that $2K_{\Xb}$ can separate any point in $X$ from any point in $D:$ 
\begin{Proposition} \label{Freeness of 2k-D}
   If $\operatorname{sys(X)}\ge 20\ln\bi(5n+(8\pi)^4\bi),$ then for every $x\in X$ there exists $s\in H^0(\Xb, 2K_{\Xb}-D)$ such that $s$ does not vanish at $x.$
\end{Proposition}
\begin{proof}
    Since $\operatorname{sys(X)}\ge 20\ln\bi(5n+(8\pi)^4\bi),$ the uniform depth of cusps is larger than $4\pi$(see Corollary  \ref{depth Bounds 2}) and therefore by Bakker-Tsimerman's result,Theorem \ref{Tsi}, it follows that $K_{\Xb}+(1-\lambda)D$ is ample for $\lambda \in (0, \frac{(n+1)d}{4\pi}).$ On the other hand, as $K_{\Xb|D}\cong -D_{D}$ and $-D_{|D}$ is ample, for every subvariety $V$ of dimension $m$ and every $1\le i\le m$ we have
    \begin{align} \nonumber
        (K_{\Xb}-2D)^{m-i} (-D)^i \cdot V_{|D} 
        &=
        - (K_{\Xb}-2D)_{|D}^i (-D_{|D})^{j-1} \cdot V\\ \nonumber
        &=
        - 3^i (-D_{|D})^{n-1} \cdot V_{|D}\\\label{Neg Intermediate}
        &\le 0.
    \end{align}
   Expanding $(K_{\Xb}-(1-\frac{(n+1)d}{4\pi})D)^m \cdot V \ge 0$ and using \eqref{Neg Intermediate} we get: 
    \begin{align*}
        (K_{\Xb}-2 D)^m \cdot V 
        &\ge
        (\frac{(n+1)d}{4\pi}+1)^{m}(-D_{|D})^{m-1} \cdot V_{|D} \\
        &\ge
        (\frac{(n+1)d}{4\pi})^m \mbox{(by the ampleness of $-D_{|D}$)}\\
 &\ge \bi(\frac{n+1}{4\pi}\bi)^m
 \cdot
e^{m\sy/16 } \mbox{ \ (by Corollary \ref{depth Bounds 2})} \\
&>
(n+1)^m n^m \mbox{ \ (by the bound on $\sy$)}
    \end{align*}
Hence, Theorem \ref{Siu Angehrn} gives that $2K_{\Xb}-D$ has a global section which does not vanish at $x.$ 
  
\end{proof}
We recall the result of Ein-Lazarsfeld-Nakamaye on the pointwise separation of jets:
\begin{theorem} (\cite[Theorem 4.4]{ein1996zero})\label{Laz}
    Let $Y$ be a smooth projective variety of dimension $n$ and let $L$ be an
ample line bundle on $Y$ satisfying 
$L^n > (n + s)^n.$ Let $b$ be a non-negative number
such that $K_Y+ bL$ is nef. Suppose that $m_0$ is a positive integer such that $m_0 L$ is
free. Then, for any point $y\in Y$ either

(a) $K_{Y}+L$ separates $s$-jets at $y$, or

(b) there exists a dimension $m$ subvariety $V$ containing $y$ and satisfying
\begin{align}
    \operatorname{deg}_{L}(V) 
    \le
    \Big(b+m_0 \cdot m+ \frac{n!}{(n-m)!}   \Big)^{n-m} (n+s)^{n} 
\end{align}
\end{theorem}

\begin{definition}\label{DefSesh}
    
Let $Y$ be a smooth projective
variety and let $L$ be a nef line bundle on $Y.$ Fix a point $y\in Y.$ The Seshadri constant of $L$ at $y$ is the real number
\begin{align*}
    \epsilon(L,y)
    =
    \operatorname{inf}\frac{L\cdot C}{\operatorname{mult}_{y}(C) },
\end{align*}
where the infimum is taken over all irreducible curves $C$ passing through $y.$
\end{definition}

Plugging in Theorem \ref{LargeIntersection} and Theorem \ref{Freeness} to the result of Ein-Lazarsfeld-Nakamaye allows us to separates $s$-jets of $2K_{\Xb}$ on $X$ if $\sy$ is sufficiently large with respect to $n$ and $s:$ 
\begin{theorem}
\label{Seperation of Jets} Let $s$ be a positive integer. Suppose that
$$\sy 
\ge 
20\operatorname{max}\{n\ln\bi((1+2n+n!)(n+s)\bi), \ln\bi(5 n+(8\pi)^4\bi) \}.$$
Then for every $x\in X,$ the line bundle $2K_{\Xb}$ separates $s$-jets at $x.$ In particular, for every $x$ we have  
$\epsilon(K_{\Xb},x)\ge s/2.$
\end{theorem}

\begin{proof}
Since $\sy >20 \ln\bi(5n+(8\pi)^4\bi),$
Theorem \ref{Freeness} implies that $2K_{\Xb}$ is globally generated.
Also, as $\sy \ge 20 \ln(n+s),$ Theorem \ref{LargeIntersection}
implies that 
\begin{align*}
    K_{\Xb}^n 
    >
    (n+s)^n.
\end{align*}
Note that substituting the lower bounds on $\sy$ from Theorem \ref{LargeIntersection} shows that for a subvariety $V$ of dimension $m$ not contained in $D$, the following inequality holds:
$$K_{\Xb}^m \cdot V 
     \ge 
     n^m (1+2n+n!)^n (n+s)^n
     \ge
     (b+2 m+\frac{n!}{(n-m)!})^{n-m} (n+s)^n. 
     $$    
Now, applying Theorem \ref{Laz} to $L=K_{\Xb}, m_0=2$ and $b=1$ gives that $2K_{\Xb}$ separates $s$-jets at every $x\in X 
 .$

 Combining the separation of jets with \cite[Proposition 2.2.5 ]{bauer2009primer} gives that 
    $
    \epsilon(2K_{\Xb},x)
    \ge
    s.
    $
    Since $\epsilon(2K_{\Xb},x)=2\epsilon(K_{\Xb},x),$ we get the desired inequality. 
\end{proof}
    
We recall a result of Kollar which tells us that a line bundle can separate two points if the degree of every subvariety passing through either of the points with respect to the line bundle is sufficiently large relative to the dimension of the ambient space:   
\begin{theorem}\label{Kollar} (\cite[Theorem 5.9]{kollar1997singularities})
    Let $L$ be a nef and big divisor on a smooth projective variety $Y$. Let
$x_1, x_2$ be closed points and assume that there are positive numbers $c(k)$ with
the following properties:
\begin{enumerate}
    \item If $V\subset Y$ is an irreducible $m$-dimensional subvariety which contains $x_1$ or $x_2$ then
$$L^m\cdot V > c(m)^m.$$
\item The numbers $c(k)$ satisfy the inequality
$$
\dis \sum_{k=1}^{\operatorname{dim}(x)} \sqrt[k]{2}\frac{k}{c(k)}
\le
1.
$$
\end{enumerate}
Then, $K_Y+L$ separates $x_1$ and $x_2.$ 
\end{theorem}

 \begin{definition}(\cite{takayama1993ample})
     Let $L$ be a line bundle on a smooth projective variety $Y$ and let $D$ be a divisor on $Y.$ The line bundle $L$ is said to be very ample modulo $D$ if the rational map $\Phi_{L}: Y \dashrightarrow \Proj(H^0(Y, O_{Y}(L)) $ is an embedding of $Y\setminus D.$
 \end{definition}
 Note that Theorem \ref{Freeness} says that the rational map 
 $\Phi_{2K_{\Xb}}: Y \dashrightarrow \Proj(H^0(Y, O_{Y}(2K_{\Xb}))$ is globally defined map on $Y.$ Moreover, the following theorem gives that this map is in particular injective on $X$ and can separate any two tangent directions at whole $\Xb:$
 \begin{theorem} \label{Very Ample bicanonical}
     Suppose that
     $$\sy 
\ge 
20\operatorname{max}\{n\ln\bi((1+2n+n!)(n+1)\bi), \ln\bi(5 n+(8\pi)^4\bi) \}.$$
     Then the map 
     $\Phi_{2K_{\Xb}}: \Xb \to \Proj\bi(H^0(\Xb,2K_{\Xb}) \bi)$
     satisfies the following properties:
     \begin{enumerate}

         \item If $\phi_{2K_{\Xb}}(x_1)=\phi_{2K_{\Xb}}(x_2)$  for some $x_1,x_2\in \Xb,$  then $x_1,x_2\in D_i,$ where $D_i$ is some connected component of $D.$ 
         \item  $\Phi_{2K_{\Xb}}$ separates tangent directions at every $x\in X.$
     \end{enumerate}

 \end{theorem}
 \begin{proof}
     Separation of points: Note that by Lemma \ref{bndry} if $\phi_{2K_{\Xb}}(x_1)=\phi_{2K_{\Xb}}(x_2)$ and $x_1,x_2\in D,$ then they both lie on the same component of $D.$ Hence, we only need to deal with the following two cases:
     \begin{enumerate}
         \item $x_1,x_2 \in X:$ Let $V\subset \Xb$ be a subvariety of dimension $m$ which passes through either $x_1$ or $x_2.$ Fix $c=ne^{\sy/20}.$ By Theorem \ref{LargeIntersection} we have that
     \begin{align*}
         K_{\Xb}^m \cdot V
         \ge
         c^m. 
     \end{align*}
     Therefore, by Kollar's Theorem, Theorem \ref{Kollar}, we can separate any two points $x_1,x_2 \in X.$  

     \item $x_1\in X, x_2\in D:$ By Proposition \ref{Freeness of 2k-D}, there is a section $s\in H^0(\Xb, 2K_{\Xb}-D)$ which does not vanish at $x_1.$  Therefore, as $2K_{\Xb}-D$ is a subbundle of $2K_{\Xb},$ we get a section of $2K_{\Xb}$ which does not vanish at $x_1,$ but vanishes on $D$ and in particular at $x_2.$   
        \end{enumerate}

     Separation of tangent directions: For $x\in X,$ the separation of tangent direction follows from Theorem \ref{Seperation of Jets} when $s=1.$
     
 \end{proof}
In particular, Theorem \ref{Very Ample bicanonical} implies that $2K_{\Xb}$ is very ample modulo $D.$

\begin{theorem}\label{Very amplness of 3k}
With the same assumption on $\sy$ as Theorem \ref{Very Ample bicanonical}, $3K_{\Xb}$ is very ample.
 \end{theorem}
\begin{proof}
    By Theorem \ref{Very Ample bicanonical}, it follows that we only need to show that $3K_{\Xb}$ can separate any two points and any tangent direction on any connected component of $D,$ which follows from Lemma \ref{Very Ampleness on Boundray}. 
\end{proof}
Putting all of these together, we get the following:
\begin{Corollary}\label{Effective 1}Suppose that
     $$\sy 
\ge 
20\operatorname{max}\{n\ln\bi((1+2n+n!)(n+1)\bi), \ln\bi(5 n+(8\pi)^4\bi) \}.$$ Then, the following hold \begin{enumerate} 
    \item $2K_{\Xb}$ is globally generated and very ample modulo $D.$
    \item $3K_{\Xb}$ is very ample.
\end{enumerate}
\end{Corollary}
\begin{proof}
The global generation of $2K_{\Xb}$ follows from 
    Theorem \ref{Freeness}. The very ampleness modulo $D$ follows from Theorem \ref{Very Ample bicanonical}. The very ampleness of $3K_{\Xb}$ follows from Theorem \ref{Very amplness of 3k}.
\end{proof}
\section{Seshadri Constant}

The goal of this section is to study the relation between the Seshadri constants and the systole of $X$ and in particular we prove Corollary \ref{full Very amplness Int} in this section.

In addition to the result of Theorem \ref{Seperation of Jets} on the Seshadri constant $\epsilon(2K_{\Xb},x)$ for $x\in X,$ we obtain the following result, which holds under a smaller bound on $\sy:$

\begin{Corollary}
    
\label{Seshadri cons locus}
Suppose that $\sy  \ge 20\ln\bi(5n+(8\pi)^4\bi).$ Let 
$$E:=\{x\in X| \epsilon(K_{\Xb},x)<e^{\sy/20} \}.$$
Then, $E$ satisfies the following properties:
\begin{enumerate}
    \item $E\cap X_{\mathrm{thick}}=\varnothing.$
    \item $E$ does not contain any positive-dimensional subvariety.
    \item $E$ is contained in a Zariski closed proper subset of $X.$
     
\end{enumerate}

\end{Corollary} 
\begin{proof}
\begin{enumerate}
    \item Fix $x\in X_{\mathrm{thick}}.$ Let $C\subset \Xb$ be a curve passing through $x.$ Since $x\in X_{\mathrm{thick}}$ we have $\operatorname{inj}_{x}(X)\ge \sy/2.$
On the other hand, since $\sy  \ge 20\ln\bi(5n+(8\pi)^4\bi)$, Corollary \ref{depth Bounds 2} gives that $d> 8\pi,$ therefore by Theorem \ref{Tsi} $K_{\Xb}-D$ is ample. We can write: 
\begin{align*}
    2K_{\Xb} \cdot C
    &\ge  
    (K_{\Xb}+D)\cdot C \mbox{\ (by ampleness of $K_{\Xb}-D$)\ }\\
    &\ge 
    \frac{n+1}{4\pi } \operatorname{vol}_{X}(C) \mbox{\ (by \eqref{ChernForm})}\\
    &\ge 
       (n+1)\sinh^{2}\bi(\sy /2\bi) \cdot \operatorname{mult}_{x}(C) \mbox{\ (by Theorem \ref{Hwanginequilty})}. 
\end{align*}
Therefore, 
$$
\epsilon(x,K_{\Xb}) 
\ge
\frac{n+1}{2 }\sinh^{2}\bi(\sy /2\bi)
>
e^{\sy/20},
$$
and this gives the first property.

\item Combining (i) with Theorem \ref{thick} we conclude that $E$ does not have any positive-dimensional subvariety.  

\item Note that Theorem \ref{LargeIntersection} implies that for every $m$-dimensional subvariety $V\not \subset D,$ we have 
$$
(K_{\Xb}^m\cdot V)^{\frac{1}{m}}
\ge 
 \frac{n+1}{4\pi } e^{\sy/16}.
$$
Putting this in \cite[Theorem 3.1]{ein1995local} gives
\begin{align} \label{seshadri bound}
 \epsilon(K_{\Xb},x)
\ge 
 \frac{1}{4\pi } e^{\sy/16}
  >
e^{\sy/20}.
\end{align}
for all $x\in \Xb$ off the union of countably many proper subvarieties
of $\Xb.$ On the other hand as $K_{\Xb}$ is ample by using \cite[Lemma 1.4]{ein1995local} we can conclude that inequality \eqref{seshadri bound} holds on Zarisiky open set, i.e, $E$ is contained in a proper subvariety
of $\Xb.$    

\end{enumerate}
\end{proof}

Consider the decomposition of the boundary divisor $D$ to the connected components $D=\sqcup_{i=1}^{k}D_i.$ Due to \cite{MOK}, we know that each $D_{i}$ is an abelian variety with ample conormal bundle $O_{D_i}(-D_i).$ The adjunction formula gives that $K_{\Xb|D_i}$ is isomorphic to the conormal bundle $O_{D_i}(-D_i).$ Suppose that $D_i=\Lambda_i \backslash W_i,$ where $W_{i}\cong \C^{n-1}$ is a complex vector space of dimension $n-1,$ and $\Lambda_i\cong \Z^{n-1}$ is a lattice in $W_{i}.$  It is classical that every ample line bundle on $D_i$ determines a positive definite Hermitian form on $W_{i}.$ 
 Suppose $H_{i}$ is the positive definite Hermitian form determined by $K_{\Xb|D_i}$ on $W_{i}.$ The real part $B_{i}=\operatorname{Re}(H_i)$
defines a Euclidean inner product on $W_i$ (see \cite[sec 5.3.A]{lazarsfeld1} for more details). Let $l_i$ be the length of a shortest vector of $\Lambda_i$ with respect to $B_{i}.$ We define the systole of the boundary as 
$$
\operatorname{sys}(D):=\operatorname{min}_{i=1}^{k} l_{i}.
$$

The following lemma gives a lower bound for the Seshadri constant of $K_{\Xb|D}$ in terms of the systole of the boundary:

\begin{Lemma} \label{SeshOnD} Let $x $ be a point on a connected component of the boundary, $ D_i.$ Then,
$$\epsilon(K_{\Xb|D_i},x) \ge \frac{\pi}{4} \cdot \operatorname{sys}(D)^{2}.$$
    
\end{Lemma}
\begin{proof}
    This follows from \cite[Theorem 5.3.6]{lazarsfeld1}.
\end{proof}

Combining this lemma with the previous results gives that if the systole of $\Xb$ and $\operatorname{sys}(D)$ are sufficiently large, then the Seshadri constant $\epsilon(K_{\Xb},x)$ is large and in particular $2K_{\Xb}$ is very ample:

\begin{Corollary}\label{full Very amplness}Suppose that $\operatorname{sys}(D)>2\sqrt{2n/\pi}$ and that $$\sy 
\ge 
20\operatorname{max}\{n\ln\bi(5n(1+2n+n!)\bi), \ln\bi(5 n+(8\pi)^4\bi) \}.$$  
 Then, for every $x\in \Xb$ we have
 $
    \epsilon(K_{\Xb},x)
    \ge 
   2n,
 $
      and in particular $2K_{\Xb}$ is very ample.  
\end{Corollary}
\begin{proof}
    Let $C\subset \Xb$ be a connected curve passing through a point $x\in \Xb.$ We consider three cases:
\begin{enumerate}
   
    \item $x\in D$ and $C$ fully contained in a $D:$ Let $D_i$ be the connected component of $D$ which contains $x.$ 
      Lemma \ref{SeshOnD} implies that 
    $$K_{\Xb} \cdot C
    =
    K_{\Xb|D_i} \cdot C  
    \ge
    \frac{\pi}{4} \operatorname{sys}(D)^2 \cdot \operatorname{mult}_{x}(C)
    \ge 
    2n\cdot \operatorname{mult}_{x}(C)
    .$$
    
    \item $x\in D$ and $C$ is not contained in $D:$ Plugging in the bound on the systole  in Theorem \ref{systolDepth} gives that the uniform depth of cusps $d$ is at least $8\pi.$ By the theorem of Bakker-Tsimerman, Theorem \ref{Tsi},
    the line bundle $K_{\Xb}+(1-\lambda)D$ is ample for $\lambda \in (0, (n+1)d/4\pi)$. Hence, we can write
    \begin{align*}
        K_{\Xb} \cdot C 
        \ge
        (\frac{(n+1)d}{4\pi}-1) D \cdot C
        \ge 
         \frac{nd}{4\pi} \operatorname{mult}_{x}(C)
        \ge2n\operatorname{mult}_{x}(C),
    \end{align*}  
because $d\ge 8\pi$ by Corollary \ref{depth Bounds 2}.

    \item $x\in X:$ For this case we will use Theorem \ref{Seperation of Jets}. Plugging in $s=2n$ to this theorem gives 
   $K_{\Xb} \cdot C \ge 2n \cdot \operatorname{mult}_{x}(C).$
    Hence, for every $x\in X$ we get that $\epsilon(K_{\Xb},x)\ge 2n.$ Combining this with Demailly's theorem \cite[Proposition 6.8).]{damailly1992singular} implies that $2K_{\Xb}$ is very ample.  
\end{enumerate}
\end{proof}
\section{Sparsity of Rational Points}
The goal of this section is to prove Corollary \ref{sparsity int}. The proof is based on Theorem \ref{Hyperbolic Vol Int}, Theorem \ref{LargeIntersection Int}, and the fundamental idea of Bombieri-Pila, known as the determinant method. 
Let $F$ be a number field with ring of integers $\mathcal{O}_F$ and set of places $M_F$.  
For each place $v\in M_F$ let $|\cdot|_v$ denote the
standard normalized absolute value on $F_v$, so that the product formula
\[
\prod_{v\in M_F} |a|_v = 1, \qquad \text{for all } a\in F^\times,
\]
holds.  
Concretely:
\begin{itemize}
  \item If $v$ is non-archimedean corresponding to a prime ideal $\mathfrak p\subset \mathcal{O}_F$, set
  \[
  |a|_v \;:=\; N(\mathfrak{p})^{-\operatorname{ord}_{\mathfrak{p}}(a)},
  \qquad a\in F^\times,
  \]
  where $N(\mathfrak{p}) = |\mathcal{O}_F/\mathfrak{p}|$ is the absolute norm of $\mathfrak{p}$.
  \item If $v$ is archimedean, arising from an embedding $\sigma:F\hookrightarrow\mathbb{R}$ or
  $\sigma:F\hookrightarrow\mathbb{C}$, set
  \[
  |a|_v := |\sigma(a)|, \qquad a\in F^\times.
  \]
  In the complex case, we include both $\sigma$ and its conjugate $\overline{\sigma}$ as distinct places.
\end{itemize}

The multiplicative projective height of a point $x=[x_0:\cdots:x_N]\in \mathbb{P}_F^N$ is defined as
\begin{equation}\label{classical multiplicative projective height}
    \operatorname{H}(x) \;=\; \prod_{v\in M_F} \max_{0\le i\le N} |x_i|_v.
\end{equation}
The product formula ensures that this definition is independent of the choice of homogeneous coordinates,
that is, scaling the representative vector $(x_0,\dots,x_N)$ by any $\lambda\in F^\times$ leaves the height unchanged.

We will use the following recent result of Maculan-Brunebarbe \cite{brunebarbe2022counting}, obtained by applying the determinant method inductively:
\begin{theorem}(\cite[Theorem 4.4]{brunebarbe2022counting})\label{counting}
    Let $Z$ be a closed subvariety of $\mathbb{P}_{F}^{N},$ let $\epsilon>0$ be a real number, let $n\ge 0$ and $e\ge 1$ be integers. 

Then, there is a real number $C=c(n,e, N, F,D,\epsilon)$ with the following property:
For an integral $n$-dimensional closed subvariety $Y$ of $\Proj^{N}$
 of degree $\le e$ such that each positive-dimensional integral closed subvariety in $Y$ not contained in $Z$ has degree $\ge d^{\operatorname{dim}(Y)}$ for some integer $d \ge 1$, and a real number $B > [F : \Q]\epsilon$, the following
inequality holds:
$$
\#\{
x\in Y(F)\setminus Z \mid \operatorname{H}(x)\le B 
\}
\le
C B^{(1+\epsilon)[F:\Q]n(n+3)/d}.$$
\end{theorem}
\begin{remark}
    There is a typo in the statement of Theorem \ref{counting} in the original paper. However, from its applications and the surrounding statements in that paper, it is clear that the power of \(d\) should be the dimension of the subvariety $Y,$ not $Z.$
\end{remark}
On $\mathbb{P}^N_F$ we fix the standard adelic metric on $\mathcal{O}_{\mathbb{P}^N}(1)$, defined as follows:  
for a local section $s$ given by a homogeneous linear polynomial and a point $x=[x_0:\cdots:x_N]\in \mathbb{P}_{F}^N$,
\[
\|s(x)\|_{v} \;=\; \frac{|s(x)|_v}{\max_{0\le i\le N}|x_i|_v}, \qquad v\in M_F.
\]

Let $\overline{Y}$ be a smooth projective variety over $F$, and let $L$ be a line bundle on $\overline{Y}$.  
Suppose there exists $b\ge 1$ such that the complete linear system $|bL|$ is base-point free.  
Choosing a basis of $H^0(\overline{Y},bL)$ defines the morphism
\[
\varphi_{|bL|} : \overline{Y} \longrightarrow \mathbb{P}^N_F.
\]
Pulling back the adelic metric on $\mathcal{O}_{\mathbb{P}^N}(1)$ gives a metric on $bL$, and we obtain a metric on $L$ by taking the $b$-th root fiberwise:  
for a local section $t$ of $L$ and a point $x\in \overline{Y}$,
\[
\|t(x)\|_{L,v} \;:=\; 
\Big(\,\|\,t^{\otimes b}(x)\,\|_{\varphi_{|bL|}^*(\mathcal{O}(1)),\,v}\,\Big)^{1/b}.
\]
The multiplicative height on $\overline{Y}(F)$ is 
\begin{align}\label{Height def}
    \operatorname{H}_L(x) \;=\; \prod_{v\in M_F} \|t(x)\|_{L,v}^{-1}.
\end{align}

\begin{remark}
\leavevmode
\begin{enumerate}
\item\label{it:indep-section}
\textbf{Independence from the choice of local section.}
Let $x\in \overline{Y}(F)$ and let $s,t$ be two local sections of $L$ defined in a neighborhood of $x$
with $s(x),t(x)\neq 0$. Since $L$ has rank $1$, there exists a rational function $f\in F(\overline{Y})^\times$
defined near $x$ such that $t = f\cdot s$. For each place $v$ we have that
\[
\|t(x)\|_{L,v} \;=\; \|f(x)\cdot s(x)\|_{L,v} \;=\; |f(x)|_v \, \|s(x)\|_{L,v}.
\]
Therefore,
\[
\prod_{v\in M_F} \|t(x)\|_{L,v}^{-1}
\;=\; \Big(\prod_{v\in M_F} |f(x)|_v^{-1}\Big)\; \Big(\prod_{v\in M_F} \|s(x)\|_{L,v}^{-1}\Big).
\]
Since $f(x)\in F^\times$, the product formula gives $\prod_{v\in M_F} |f(x)|_v = 1$, hence
\[
\prod_{v\in M_F} \|t(x)\|_{L,v}^{-1} \;=\; \prod_{v\in M_F} \|s(x)\|_{L,v}^{-1}.
\]
Thus $\operatorname{H}_L(x)(x)$ is independent of the chosen local section.

\item\label{it:recovers-projective}
\textbf{Recovery of the usual height on $\mathcal{O}_{\mathbb{P}_F^N}(1)$.}
Let $\overline{Y}=\mathbb{P}_F^N$ and $L=\mathcal{O}_{\mathbb{P}^N}(1)$.  
For $x=[x_0:\cdots:x_N]\in \mathbb{P}_F^N$, choose an index $j$ with $x_j\neq 0$ and take the section $s = X_j$ (the $j$-th coordinate function), which does not vanish at $x$.
By definition of the metric at every place $v$,
\[
\|s(x)\|_{v} \;=\; \frac{|X_j(x)|_v}{\max_{i}|x_i|_v} \;=\; \frac{|x_j|_v}{\max_{i}|x_i|_v}.
\]
Hence
\[
\operatorname{H}_{L}(x)
\;=\; \prod_{v\in M_F} \|s(x)\|_{v}^{-1}
\;=\; \prod_{v\in M_F} \frac{\max_i |x_i|_v}{|x_j|_v}
\;=\; \Big(\prod_{v\in M_F} \max_i |x_i|_v\Big)\cdot
      \Big(\prod_{v\in M_F} |x_j|_v^{-1}\Big).
\]
By the product formula $\prod_{v} |x_j|_v = 1$, so
\[
\operatorname{H}_{L}(x) \;=\; \prod_{v\in M_F} \max_{0\le i\le N} |x_i|_v,
\]
which is exactly the classical multiplicative projective height \eqref{classical multiplicative projective height}.

\item \textbf{Tensor powers.}
If $a\ge 1$, then the induced metric on $aL$ is obtained by taking the $a$-th power of the norm, so for every $x\in \overline{Y}(F)$,
\[
\|t^{\otimes a}(x)\|_{aL,v} = \|t(x)\|_{L,v}^a,
\]
for every local section $t$ of $L$. Therefore
\[
\operatorname{H}_{aL}(x)=\operatorname{H}_{L}(x)^a.
\]
\end{enumerate}
\end{remark}

Now, combining our effective estimates (Corollary \ref{Deg of log} and Theorem \ref{LargeIntersection}) on the degree of the subvarieties with Theorem \ref{counting} we can conclude the following:
\begin{Corollary}\label{sparsity} Suppose $\Xb$ is defined on the number field $F.$ Let $\epsilon$ be a positive number and $B$ any number such that $B\ge \epsilon[F:\Q].$
\begin{enumerate} 

    \item Let $L_1=K_{\Xb}+D.$ Then, there exists a constant $c_1$ depending on $X, F$ and $\epsilon$ such that:
\begin{align*}
         \#\Big\{
x\in X(F) \mid \operatorname{H}_{L_1}(x)\le B\Big\} \le c_1 B^{\delta},
     \end{align*}
where $$\delta=\frac{[F:\Q] n(n+3)}{ \sinh^{2}\bi(\sy /2\bi) (n+1) }(1+\epsilon),$$ 
and $H_{L_1}$ is the multiplicative height. 

\item Let $L_2=K_{\Xb}$ and assume that $\sy \ge 4\ln\bi(5n+(4\pi)^4\bi).$ Then, there exists a constant $c_2$ depending on $X, F$ and $\epsilon$ such that
\begin{align*}
         \#\Big\{
x\in X(F) \mid \operatorname{H}_{L_2}(x)\le B\Big\} \le c_2 B^{\delta},
     \end{align*}
where $$\delta=\frac{4\pi [F:\Q] (n+3)}{ e^{\sy/16}}(1+\epsilon),$$ 
and $H_{L_2}$ is the multiplicative height. 

\end{enumerate}
        
\end{Corollary}
\begin{proof}
\begin{enumerate}
    \item By the main Theorem of \cite{MOK}, there exists $b$ such that $bL_1$ is base-point free on $\Xb$ and it embeds $X$ into some projective space $\Proj^N$ such that each connected component of $D$ collapses to an isolated point. Let $Z$ be the union of these isolated points in $\Proj^N$.  Applying Corollary \ref{Deg of log} implies that for every subvariety $V$ of $\Xb$ not contained in $D$ one has:
      $$
        ((bL_1)^m \cdot V) ^{1/m} \ge b (n+1)\sinh^{2}\bi(\sy /2\bi),
     $$
where $m$ is the dimension of $V.$ Hence, applying Theorem \ref{counting} gives us that:
     \begin{align*}
         \#\{
x\in X(F) \mid \operatorname{H}_{bL_1}(x)\le B\} \le C B^{ [F:\Q] n(n+3)(1+\epsilon)/(n+1)s},
     \end{align*}
     where $s= \sinh^2(\sy/2)$ and $C$ is constant depending on $X, F$ and $\epsilon$ (Note that $N, n$ and $e$ are fixed when we fixed $\Xb$ and $bL_1.$ Also, the toroidal compactification is unique for a ball quotient, therefore all of these data only depend on $X$). To conclude, note that $\operatorname{H}_{L_1}(x)\le B$ if and only if $\operatorname{H}_{bL_1}(x)\le B^b.$ Therefore replacing $B$ with $B^b$ implies the claim.

    \item  We will proceed similar to the previous part, the only difference is that we use the embedding with multiple of $L_2$ instead of $L_1.$ With the bound on the systole, Theorem \ref{LargeIntersection} tells us that $L$ is an ample bundle as it has positive intersection with all subvarieties. Let $b$ be an integer such that $bL_2$ is very ample. Now, we can embed $\Xb$ into some projective space $\Proj^N$ by $bL_2.$ Applying Theorem \ref{LargeIntersection} gives us that for every subvariety of $\Xb$ not contained in $D$ one has:
     $$
        ((bL_2)^m \cdot V) ^{1/m} \ge \bi(\frac{nb}{4\pi}\bi) e^{\sy/16}.
     $$

     Applying Theorem \ref{counting} gives us that:
     \begin{align*}
         \#\{
x\in X(F) \mid \operatorname{H}_{bL_2}(x)\le B\} \le C B^{4\pi [F:\Q] (n+3)(1+\epsilon)/bs},
     \end{align*}
where $s= e^{\sy/16},$ and $C$ is constant depending on $X, F$ and $\epsilon.$ Similar to the previous part we can conclude the desired inequality.

\end{enumerate}
\end{proof}

\bibliographystyle{alpha}
\bibliography{Ref}

\end{document}